\documentclass[10pt,reqno]{amsart}

\usepackage{amsmath,amsfonts,amsbsy,amsgen,amscd,mathrsfs,amssymb,amsthm}
\usepackage{enumitem,tikz}

\newtheorem{thm}{Theorem}[section]
\newtheorem{lem}[thm]{Lemma}

\newtheorem{prop}[thm]{Proposition}
\newtheorem{cor}[thm]{Corollary}

\theoremstyle{definition}

\newtheorem{defn}[thm]{Definition}
\newtheorem{example}[thm]{Example}
\newtheorem{rem}[thm]{Remark}

\numberwithin{equation}{section} 
\numberwithin{figure}{section}
\numberwithin{table}{section}

\newcommand{\card}{\mathop{\mathrm{card}}}

\begin{document}

\title[Nonhomogeneous random matrices]{The dimension-free structure of \\
nonhomogeneous random matrices}

\author{Rafa{\l} Lata{\l}a}
\address{Institute of Mathematics, University of Warsaw, Banacha 2, 
02-097, Warsaw, Poland}
\email{rlatala@mimuw.edu.pl}

\author{Ramon van Handel}
\address{Fine Hall 207, Princeton University, Princeton, NJ 
08544, USA}
\email{rvan@princeton.edu}

\author{Pierre Youssef}
\address{Laboratoire de Probabilit\'es, Statistique et Mod\'elisation,
Universit\'e Paris-Diderot, 5 rue Thomas Mann, 75205, Paris CEDEX 13,   
France}
\email{youssef@lpsm.paris}

\begin{abstract}
Let $X$ be a symmetric random matrix with independent but non-identically
distributed centered Gaussian entries. We show that 
$$
        \mathbf{E}\|X\|_{S_p} \asymp
        \mathbf{E}\Bigg[
        \Bigg(\sum_i\Bigg(\sum_j X_{ij}^2\Bigg)^{p/2}\Bigg)^{1/p}
        \Bigg]
$$
for any $2\le p\le\infty$, where $S_p$ denotes the $p$-Schatten class and 
the constants are universal. 
The right-hand side admits an explicit expression in terms of the 
variances of the matrix entries. This settles, in the case $p=\infty$, a 
conjecture of the first author, and provides a complete 
characterization of the class of infinite matrices with 
independent Gaussian entries that define bounded operators on $\ell_2$.  
Along the way, we obtain optimal dimension-free bounds on the moments 
$(\mathbf{E}\|X\|_{S_p}^p)^{1/p}$ that are of independent interest.
We develop further extensions to non-symmetric matrices and to
nonasymptotic moment and norm estimates for matrices with non-Gaussian 
entries that arise, for example, in the study of random graphs and
in applied mathematics.
\end{abstract}



\maketitle

\thispagestyle{empty}

\section{Introduction and main results}

The study of random matrices has long driven mathematical developments 
across a range of pure and applied mathematical disciplines. Initially 
motivated by questions arising in mathematical physics, classical random 
matrix theory (see, e.g., \cite{Tao12} for an introduction to this 
topic) is primarily concerned with random matrix models that possess a 
large degree of symmetry. For example, perhaps the most basic objects in 
this theory are matrices with independent and identically distributed 
entries, called Wigner matrices. Major advances on this subject were 
achieved in the past decade, resulting in an extremely detailed 
understanding of the fine-scale properties of the spectra of large Wigner 
and Wigner-like matrices.

There are however many situations in which classical random matrix models 
are of limited significance. For example, from a functional analytic 
perspective, one might naturally wish to view a random matrix as a random 
linear operator. One of the most basic questions one could ask in this 
context is under what conditions an (infinite) random matrix defines a 
bounded operator on $\ell_2$. Even this simple question appears, at 
present, to be almost entirely open. It is readily verified that such 
matrices could never have identically distributed entries; to obtain 
meaningful answers to such infinite-dimensional questions, it is therefore 
essential to consider nonhomogeneous random matrix models. In another 
context, many problems of applied mathematics, such as the analysis of 
random networks or numerical linear algebra, naturally give rise to 
structured random matrix models that are inherently nonhomogeneous. Such 
problems motivate the development of mathematical methods that can 
accurately capture the underlying structure.

The main approach to the study of general nonhomogeneous random matrices 
has been provided by variations on a classical result in noncommutative 
probability, the noncommutative Khintchine inequality of Lust-Piquard and 
Pisier \cite[section 6]{PX03}. We describe this inequality for 
concreteness in the setting of Gaussian symmetric matrices, though 
variants of it may be developed in much greater generality. Let $X$ be any 
symmetric random matrix with centered jointly Gaussian entries. It is 
readily verified that such a matrix can always be represented as 
$X=\sum_{i\ge 1}g_iA_i$, where $g_i$ are i.i.d.\ standard Gaussian 
variables and $A_i$ are given symmetric matrices. The noncommutative 
Khintchine inequality states that the moments of $X$ admit 
essentially the same estimates that would hold if $A_i$ were scalar 
quantities, that is, we have the following matrix analogue of the 
classical Khintchine inequalities:
$$
	\Bigg\|\Bigg(\sum_{i\ge 1}A_i^2\Bigg)^{\frac{1}{2}}\Bigg\|_{S_{p}} 
	\lesssim
	\mathbf{E}\|X\|_{S_p} \le
	(\mathbf{E}\|X\|_{S_p}^p)^{1/p} \lesssim
	\sqrt{p}\,
	\Bigg\|\Bigg(\sum_{i\ge 1}A_i^2\Bigg)^{\frac{1}{2}}\Bigg\|_{S_{p}}
$$
for all $1\le p<\infty$, where we denote by 
$\|X\|_{S_p}:=\mathrm{Tr}[|X|^p]^{1/p}$ the Schatten $p$-norm (that is, 
the $\ell_p$-norm of the singular values of $X$). This powerful estimate 
makes no assumption whatsoever on the covariance structure of the matrix 
entries; consequently, this result and its generalizations have 
had a major impact in noncommutative probability \cite{PX03} as well as in 
applied mathematics \cite{Tro15}.

Despite the significant power of the noncommutative Khintchine inequality, 
its conclusion remains in many ways unsatisfactory: both the upper and 
lower bounds become increasingly inaccurate for large $p$. For example, 
while this result characterizes, for fixed $p<\infty$, when an infinite 
random matrix is in the Schatten class $S_p$, it sheds little light on the 
question of which infinite random matrices define bounded operators on 
$\ell_2$ (the case $p=\infty$). In finite dimension $n$, one can still 
deduce useful quantitative bounds on the operator norm using that 
$\|X\|_{S_p}\asymp\|X\|_{S_{\infty}}$ for $p\sim\log n$. The resulting 
dimension-dependent bounds are notoriously inaccurate, however: they do 
not even capture correctly the norm of the most basic object in random 
matrix theory, the Gaussian Wigner matrix. These observations indicate 
that a detailed understanding of the spectral norms of nonhomogeneous 
random matrices will require far more precise information than is provided 
by the noncommutative Khintchine inequality. In the most general setting 
considered thus far, this aim remains out of reach. However, in 
this paper, we will settle these questions in what is perhaps the most 
natural case: that of nonhomogeneous random matrices with 
independent Gaussian (and some non-Gaussian) entries.

Let us now consider, therefore, any symmetric random matrix $X$ with 
independent centered Gaussian entries. Such a matrix may be 
represented as $X_{ij}=b_{ij}g_{ij}$, where $g_{ij}$ are i.i.d.\ standard 
Gaussian variables and $b_{ij}\ge 0$ are given scalars for $i\ge j$. In 
this case, the noncommutative Khintchine inequality reduces to
$$
	\Bigg(\sum_i\Bigg(\sum_j b_{ij}^2\Bigg)^{\frac{p}{2}}\Bigg)^{\frac{1}{p}}
	\lesssim \mathbf{E}\|X\|_{S_p} \le
	(\mathbf{E}\|X\|_{S_p}^p)^{1/p}
	\lesssim
	\sqrt{p}\,
	\Bigg(\sum_i\Bigg(\sum_j b_{ij}^2\Bigg)^{\frac{p}{2}}\Bigg)^{\frac{1}{p}}.
$$
One of the main results of this paper is the following sharp form 
of this estimate.

\begin{thm}
\label{thm:main}
Let $X$ be an $n\times n$ symmetric matrix with 
$X_{ij}=b_{ij}g_{ij}$, where $b_{ij}\ge 0$ and $g_{ij}$ are i.i.d.\ 
standard Gaussian variables for $i\ge j$. Then
$$
	(\mathbf{E}\|X\|_{S_p}^p)^{1/p} \asymp
	\Bigg(\sum_i\Bigg(\sum_j b_{ij}^2\Bigg)^{\frac{p}{2}}\Bigg)^{\frac{1}{p}}
	+ \sqrt{p}\,\Bigg(\sum_{i}\max_j b_{ij}^{p}\Bigg)^{\frac{1}{p}}
$$
and
$$
	\mathbf{E}\|X\|_{S_p} 
	\asymp
	\Bigg(\sum_i\Bigg(\sum_j b_{ij}^2\Bigg)^{\frac{p}{2}}\Bigg)^{\frac{1}{p}} +
	\max_{i\le e^p}\max_j b_{ij}^*\sqrt{\log i} +
	\sqrt{p}\,\Bigg(
	\sum_{i\ge e^p}\max_j {b_{ij}^{*}}^p
	\Bigg)^{\frac{1}{p}}
$$
for all $2\le p<\infty$, and
$$
	\mathbf{E}\|X\|_{S_\infty}
	\asymp
	\max_i\sqrt{\sum_j b_{ij}^2} + \max_{ij}b_{ij}^*\sqrt{\log i}
$$
for $p=\infty$.
Here the matrix $(b_{ij}^*)$ is obtained by permuting the rows and 
columns of the matrix $(b_{ij})$ such that
$\max_j b_{1j}^* \ge \max_j b_{2j}^* \ge \cdots \ge \max_j b_{nj}^*$,
and the constants in the estimates are universal \emph{(}independent
of $n,p,\{b_{ij}\}$\emph{)}.
\end{thm}

As a consequence we obtain, for example, a characterization of all 
infinite matrices with independent Gaussian entries that define bounded 
operators on $\ell_2$.

\begin{cor}
\label{cor:inf}
Let $(X_{ij})_{i,j\in\mathbb{N}}$ be a symmetric infinite matrix with 
independent Gaussian entries 
$X_{ij}\sim N(a_{ij},b_{ij}^2)$ for $i\ge j$. We have the following
dichotomy:
\begin{itemize}[label=\textbullet, leftmargin=*]
\itemsep\abovedisplayskip
\item If
$$
	\max_i\sum_j b_{ij}^2 <\infty,\qquad
	\max_{ij}b_{ij}^*\sqrt{\log i}<\infty,\qquad
	\|(a_{ij})\|_{S_\infty}<\infty,
$$
then $X$ defines a bounded operator on $\ell_2(\mathbb{N})$ a.s.
\item Otherwise, $X$ is unbounded as an operator on $\ell_2(\mathbb{N})$ 
a.s.
\end{itemize}
\end{cor}

We will also develop a number of extensions of Theorem \ref{thm:main} to 
non-symmetric matrices and to matrices with non-Gaussian entries.

\begin{rem}
While Theorem \ref{thm:main} is formulated for centered (zero 
mean) Gaussian matrices, 
the result is easily extended to noncentered matrices by noting that
$$
	\mathbf{E}[\|A+X\|_{S_p}^q]^{1/q} \asymp \|A\|_{S_p} + 
	\mathbf{E}[\|X\|_{S_p}^q]^{1/q}
$$
for any deterministic matrix $A$ and $p,q$ (see Remark 
\ref{rem:nonzeromean} below). There is therefore no loss of generality in 
restricting attention to centered random matrices, as we will do for 
simplicity throughout the remainder of the paper.
\end{rem}

At first sight, the statement of Theorem \ref{thm:main} may appear 
difficult to interpret. In fact, as we will presently explain, this result 
has a very simple probabilistic formulation that sheds significant light 
on the behavior of these matrices. Moreover, its proof will provide 
considerable insight into the structure of such matrices.

\subsection{Probabilistic formulation}

The questions investigated in this paper have their origin in a result of 
Seginer \cite{Seg00}, which provides matching probabilistic upper and 
lower bounds on the operator norm of random matrices with independent and 
\emph{identically distributed} entries with an arbitrary centered 
distribution: the expected operator norm of any such matrix is of the same 
order as the expectation of the maximal Euclidean norm of its rows and 
columns. The latter is a much simpler probabilistic quantity, which can be 
readily estimated in practice. The distribution of the entries is entirely 
irrelevant here: Seginer only uses that the distribution of the matrix is 
invariant under permutation of the entries, so that one may reduce the 
problem to a question about random matrices defined by random 
permutations. The latter can be investigated by combinatorial methods.

In view of such a remarkably general probabilistic principle, it is 
natural to ask whether the same conclusion also extends to nonhomogeneous 
matrix models that do not possess permutation symmetry. Unfortunately, 
this is not the case: as was already noted by Seginer, his result for 
i.i.d.\ matrices already fails to extend to the simplest examples of 
nonhomogeneous matrices with subgaussian entries. Surprisingly, however, 
no counterexamples could be found to the analogous question for 
\emph{Gaussian} matrices. This led the first author to conjecture, about 
15 years ago, that Seginer's conclusion might remain valid in 
general for nonhomogeneous random matrices with independent centered 
Gaussian entries (cf.\ \cite{Lat05}). Theorem \ref{thm:main} settles 
this conjecture in the affirmative: indeed, one may compute \cite[Theorem 
1.2]{vH17}
$$
	\mathbf{E}\Bigg[
	\max_i\sqrt{\sum_j X_{ij}^2}
	\Bigg] \asymp
	\max_i\sqrt{\sum_j b_{ij}^2} +
	\max_{ij} b_{ij}^*\sqrt{\log i},
$$
so that this quantity is always of the same order as 
$\mathbf{E}\|X\|_{S_\infty}$ by Theorem \ref{thm:main}. Beside settling 
the conjecture, this observation furnishes the case $p=\infty$ of Theorem 
\ref{thm:main} with a natural probabilistic interpretation that is far 
from evident from the explicit expression. Let us note that, on the face 
of it, this conclusion is quite striking: it is trivial that the 
operator norm of a matrix must be large if the matrix possesses a row with 
large Euclidean norm; what we have shown is that for symmetric Gaussian 
matrices with independent centered entries, this is the \emph{only} reason 
why the operator norm can be large, regardless of the variance pattern of 
the matrix entries. For some further discussion and a different 
probabilistic interpretation, see \cite{vH17}.

It turns out that the above probabilistic formulation can be developed 
in much greater generality. To this end, define the mixed norm
$$
	\|X\|_{\ell_p(\ell_2)} :=
	\Bigg(\sum_i\Bigg(\sum_j X_{ij}^2
	\Bigg)^{p/2}\Bigg)^{1/p}
$$
(the definition is extended in the obvious manner to the case $p=\infty$).
As a consequence of Theorem \ref{thm:main}, we will show that the
distributions of the random variables $\|X\|_{S_p}$ and
$\|X\|_{\ell_p(\ell_2)}$ are comparable in a strong sense.

\begin{cor}
\label{cor:schattentail}
Under the assumptions of Theorem \ref{thm:main}, we have
$$
	\mathbf{P}[\|X\|_{\ell_p(\ell_2)}\ge t] \le
	\mathbf{P}[\|X\|_{S_p}\ge t] \le
	C\,\mathbf{P}[\|X\|_{\ell_p(\ell_2)}\ge t/C]
$$
for all $t\ge 0$ and $2\le p\le\infty$, where $C$ is a universal constant.
In particular,
$$
	\mathbf{E}\Psi(\|X\|_{\ell_p(\ell_2)}) \le
	\mathbf{E}\Psi(\|X\|_{S_p}) \le
	\mathbf{E}\Psi(C\|X\|_{\ell_p(\ell_2)})
$$
for every increasing convex function $\Psi$ and $2\le p\le\infty$.
\end{cor}

By taking $\Psi(x)=x^p$ or $\Psi(x)=x$, respectively, this result provides 
a tantalizing probabilistic interpretation of both explicit bounds that 
appear in Theorem \ref{thm:main}: what we have shown is that for any 
symmetric Gaussian matrix with independent centered entries, the Schatten 
$p$-norm (that is, the $\ell_p$-norm of its eigenvalues) is of the same 
order as the $\ell_p$-norm of the Euclidean norms of its rows.

\subsection{Main ideas of the proof}

The proof of Theorem \ref{thm:main} builds on initial progress that was 
made in this direction in two earlier papers \cite{BvH16,vH17}. 

In 
\cite[Theorem 1.1]{BvH16}, Bandeira and the second author proved the 
following bound on the operator norm of a Gaussian random matrix, which 
corresponds to the special case $p\sim\log n$ of the first 
inequality of Theorem \ref{thm:main} (cf.\ Theorem \ref{thm:bvh} below):
$$
	(\mathbf{E}\|X\|_{S_{\infty}}^{\log n})^{1/\log n}
	\asymp
	\max_i\sqrt{\sum_j b_{ij}^2} +
	\max_{ij} b_{ij}\sqrt{\log n}.	
$$
The proof relies on a comparison principle between the moments
$\mathbf{E}[\mathrm{Tr}[X^{2p}]]$ of the $n$-dimensional matrix $X$ and 
the moments $\mathbf{E}[\mathrm{Tr}[Y^{2p}]]$ of an $r$-dimensional Wigner 
matrix $Y$, where $r$ depends only on the coefficients $b_{ij}$ and $p$.
This elementary dimension-compression argument has proved to yield a very 
efficient mechanism to reduce questions about nonhomogeneous random 
matrices to analogous questions on homogeneous random matrices, which are 
already well understood.

By Jensen's inequality, the above result provides a dimension-dependent 
upper bound on the expected operator norm $\mathbf{E}\|X\|_{S_\infty}$ 
that turns out to be nearly sharp in many cases of interest. 
Nonetheless, the dimension-dependence makes this bound useless in 
infinite-dimensional situations, while our main aim is to obtain optimal 
dimension-free bounds. Unfortunately, this bound already yields the 
optimal result that could be obtained for the operator norm by the moment 
method.
In order to surmount this obstacle, the second author developed in 
\cite{vH17} an entirely different method to bound the operator norm of 
nonhomogeneous Gaussian matrices through the geometry of random processes.
This approach made it possible to prove the following dimension-free 
bound, cf.\ \cite[Theorem 1.4]{vH17} and subsequent discussion:
$$
        \mathbf{E}\|X\|_{S_\infty}\lesssim
        \max_i\sqrt{\sum_j b_{ij}^2} +
        \max_{ij}b_{ij}^*\log i.
$$
This result is strongly reminiscent of the sharp bound provided by
Theorem \ref{thm:main} in the case $p=\infty$, but nonetheless falls short 
of this goal: to obtain the optimal estimate, the factor $\log i$ in the 
second term should be reduced to $\sqrt{\log i}$. While this may seem
a small step away from the correct result, the suboptimal term appears 
to arise from a fundamental obstruction in this method of 
proof, cf.\ \cite[section 4.3]{vH17}. It is therefore 
unclear how this argument could be significantly improved.

In view of the two suboptimal bounds discussed above, it is natural to 
expect that the optimal result should arise by eliminating an inefficiency 
in the proof of one of these bounds. One of the main insights of this 
paper is that the missing ingredient for the proof of Theorem 
\ref{thm:main} (in the case $p=\infty$) lies in an entirely different 
place: to prove this result, we will exploit some strong structural 
information on the variance pattern of the underlying matrix. To this end, 
a key idea that we will develop is that, modulo a relabeling of the rows 
and columns, the random matrix $X$ can have bounded operator norm only if 
has a very specific form: it consists of a (nearly) block-diagonal 
``core'', which is made up of blocks of controlled dimension in which the 
variances of the entries decay at a slow rate; and an off-diagonal 
remainder, in which the entry variances decay at a much faster rate (this 
structure will be made precise in Lemma \ref{lem:core} and is illustrated 
in Figure \ref{fig:opdecomp} below). Remarkably, it turns out that with 
this structure in hand, the suboptimal bounds described above already 
suffice to conclude the proof for $p=\infty$: the dimension-dependent 
bound optimally controls the diagonal blocks, which are low-dimensional by 
construction; while the dimension-free bound is sufficiently accurate to 
optimally control the remainder of the matrix. Beside enabling us to prove 
the result, this structural information provides considerable insight into 
what random matrices with bounded operator norm look like: they must be 
nearly block-diagonal, in precise sense.

For the general case where $2\le p<\infty$, we will introduce a further 
decomposition of the matrix (illustrated in Figure \ref{fig:spdecomp}): in 
one part of the matrix, the $S_p$-norm is of the same order as the 
$S_\infty$-norm, which was already addressed above; while in the remaining 
part, the expected $S_p$-norm is optimally captured by the corresponding 
moment bound. The main difficulty in this part of the proof is to obtain 
optimal bounds on the moments $(\mathbf{E}\|X\|_{S_p}^p)^{1/p}$ for 
$p\ll\log n$, which do not follow from the simple argument of 
\cite{BvH16}. While the proof is based on the same dimension-compression 
idea as in \cite{BvH16}, the optimal implementation of this method is 
significantly more challenging and requires us to develop certain 
combinatorial ``Brascamp-Lieb-type'' inequalities that may be of 
independent interest (see section \ref{sec:holder} below).

Let us note that the proofs of the case $p=\infty$ of Theorem 
\ref{thm:main} and of Corollaries \ref{cor:inf} and \ref{cor:schattentail} 
can be read independently of the more technical combinatorial arguments 
needed to extend our results to general Schatten norms. The reader who is 
interested primarily in the operator norm may skip ahead directly to 
section \ref{sec:opnorm}.

\subsection{Non-symmetric matrices and non-Gaussian entries}

So far we have stated our results in the setting of symmetric matrices 
with independent centered Gaussian entries. However, neither 
symmetry nor Gaussianity is essential, and we will develop various 
extensions of our main results to more general situations.

The easiest to eliminate is the symmetry assumption, which plays little 
role in the proofs and is largely made for notational convenience. All our 
results extend readily to non-symmetric matrices, as will be shown in 
section \ref{sec:nonsym}.

The question of non-Gaussian entries is more subtle. As was explained 
above, an example of Seginer \cite[Theorem 3.2]{Seg00} shows that the 
conclusion of Corollary~\ref{cor:schattentail} \emph{cannot} hold even for 
subgaussian variables. Despite superficial similarities, Corollary 
\ref{cor:schattentail} arises for an entirely different reason than 
Seginer's result for i.i.d.\ matrices: we cannot prove a general 
comparison principle between $\|X\|_{S_p}$ and $\|X\|_{\ell_p(\ell_2)}$; 
instead, we prove explicit upper and lower bounds on each of these 
quantities in the Gaussian case, and show that these coincide. Gaussian 
analysis plays a crucial role in the proofs of these bounds, which 
suggests that the Gaussian distribution is rather special. Nonetheless, we 
will show in section \ref{sec:nong} that Corollary \ref{cor:schattentail} 
extends to a large class of non-Gaussian random matrices. Roughly 
speaking, we expect that the equivalence between $\|X\|_{S_p}$ and 
$\|X\|_{\ell_p(\ell_2)}$ should hold rather generally for entry 
distributions whose tails are at least as heavy as that of the Gaussian 
distribution, but not when the tails are strictly lighter than the 
Gaussian distribution. We will not develop this statement in full 
generality, but rather demonstrate this phenomenon in a broad class of 
heavy-tailed distributions.

While our results cannot yield optimal two-sided bounds for matrices with 
light-tailed entries, the explicit bounds of Theorem \ref{thm:main} still 
yield very useful \emph{upper} bounds in this case. Of particular interest 
in this setting is the fact that, while Theorem \ref{thm:main} is sharp 
only up to universal constants, the sharpest moment bounds obtained in 
this paper possess nearly optimal constants. We will exploit this 
fact in section~\ref{sec:bdd} to obtain some very sharp bounds on the 
norms of non-homogeneous matrices with bounded entries. For example, we 
will prove the following result:

\begin{thm}
\label{thm:bmain}
Let $X$ be an $n\times n$ symmetric matrix with independent, 
centered, and uniformly bounded entries for $i\ge j$. Then we have for 
every $p\in\mathbb{N}$
$$
	(\mathbf{E}\|X\|_{S_{2p}}^{2p})^{1/2p} \le
	2\Bigg(\sum_i\Bigg(\sum_j 
	\mathbf{E}[X_{ij}^2]\Bigg)^p\Bigg)^{\frac{1}{2p}}
	+C\sqrt{p}\,
	\Bigg(\sum_{i,j}
	\|X_{ij}\|_\infty^{2p}\Bigg)^{\frac{1}{2p}},
$$
where $C$ is a universal constant.
\end{thm}

This improves a result proved in \cite{BvH16} for the special case 
$p\sim\log n$ under an additional symmetry assumption. We will spell out 
several variations on this result in section \ref{sec:bdd} in view of 
their considerable utility in applications. For example, we will show that 
Theorem \ref{thm:bmain} effortlessly yields the correct behavior of the 
spectral edge of the centered adjacency matrix of Erd\H{o}s-R\'enyi random 
graphs, recovering a recent result of \cite{BBK17} by a simpler and more 
general method.

\subsection{Overview and notation}
\label{sec:notation}

The rest of this paper is organized as follows. In section 
\ref{sec:moment}, we obtain optimal bounds on the moments 
$(\mathbf{E}\|X\|_{S_p}^p)^{1/p}$ in the Gaussian symmetric case, proving 
the first part of Theorem \ref{thm:main}. The second part of Theorem 
\ref{thm:main}, concerning the expected norms $\mathbf{E}\|X\|_{S_p}$, is 
proved in section \ref{sec:norm}. Here we also prove Corollary 
\ref{cor:schattentail}, as well as the characterization of infinite 
Gaussian matrices that define bounded operators on $\ell_2$. Finally, in 
section \ref{sec:ext}, we develop various extensions of our main results 
to non-symmetric and non-Gaussian matrices.

The following notation and terminology will be used throughout the paper. 
The Schatten norm $\|X\|_{S_p}$ and mixed norm $\|X\|_{\ell_p(\ell_2)}$ 
were already defined above. We will sometimes denote the operator norm as 
$\|X\|:=\|X\|_{S_\infty}$ for notational simplicity. All random matrices 
will be real unless otherwise noted (however, the extension of our main 
results to complex matrices is essentially trivial, cf.\ section 
\ref{sec:complex} below). We write $a\lesssim 
b$ or $b\gtrsim a$ to denote that $a\le Cb$ for a universal constant $C$ 
(which does not depend on any parameter of the problem unless explicitly 
noted otherwise). We will write $a\asymp b$ if $a\lesssim b$ and 
$b\lesssim a$. Finally, we recall that a random variable $Z$ is said 
to be $\sigma$-subgaussian if $\mathbf{E}[e^{tZ}] \le e^{t^2\sigma^2/2}$ 
for all $t\in\mathbb{R}$.

\section{Moment estimates}
\label{sec:moment}

The main result of this section is the following theorem.

\begin{thm}
\label{thm:moment}
Let $X$ be an $n\times n$ symmetric matrix with 
$X_{ij}=b_{ij}g_{ij}$, where $b_{ij}\ge 0$ and $g_{ij}$ are i.i.d.\ 
standard Gaussian variables for $i\ge j$. Then for $p\in\mathbb{N}$
$$
	(\mathbf{E}[\mathrm{Tr}[X^{2p}]])^{1/2p} \le
	2\Bigg(\sum_i\Bigg(\sum_j b_{ij}^2\Bigg)^p\Bigg)^{1/2p}
	+ 5\sqrt{2p}\,\Bigg(\sum_i \max_j b_{ij}^{2p}\Bigg)^{1/2p}.
$$
\end{thm}

This result essentially contains the first part of Theorem \ref{thm:main}; 
it remains to extend the conclusion to arbitrary (non-integer) $p$ and to 
prove a matching lower bound. To this end, we will establish the following 
corollary:

\begin{cor}
\label{cor:moment}
Under the assumptions of Theorem \ref{thm:moment}, we have
for all $2\le p<\infty$
$$
	(\mathbf{E}\|X\|_{S_p}^p)^{1/p} \asymp
	\Bigg(\sum_i\Bigg(\sum_j b_{ij}^2\Bigg)^{p/2}\Bigg)^{1/p}
	+ \sqrt{p}\,\Bigg(\sum_{i,j} b_{ij}^{p}\Bigg)^{1/p}.
$$
\end{cor}

Together with Remark \ref{rem:strange} below, this concludes the first
part of Theorem \ref{thm:main}.
The remainder of this section is devoted to the proofs of these results.

The generality of Corollary \ref{cor:moment} comes at the expense of 
replacing the sharp constants of Theorem \ref{thm:moment} by larger 
universal constants. Theorem \ref{thm:moment} is therefore of independent 
interest, and will be put to good use in section \ref{sec:nong} below.

\begin{rem}
\label{rem:strange}
It may appear somewhat strange that the second term in the upper bound of 
Theorem \ref{thm:moment} is smaller than the second term in the lower 
bound of Corollary \ref{cor:moment}. It is instructive to give a direct 
proof that the two bounds are nonetheless of the same order. To this end, 
we can estimate using Young's inequality
$$
	\sum_{i,j}b_{ij}^{2p} \le
	\sum_i \Bigg(\sum_j b_{ij}^2\Bigg)
	\max_j b_{ij}^{2p-2} \le
	\frac{a^p}{p}\sum_i\Bigg(\sum_j b_{ij}^2\Bigg)^p
	+
	\frac{p-1}{pa^{p/(p-1)}}
	\sum_i \max_j b_{ij}^{2p}
$$
for any $a>0$. Therefore, setting $a=e^{-2(p-1)}$ we readily obtain
$$
	\Bigg(\sum_{i,j}b_{ij}^{2p}\Bigg)^{1/2p} \lesssim
	e^{-p}
	\Bigg(\sum_i\Bigg(\sum_j b_{ij}^2\Bigg)^p\Bigg)^{1/2p}
	+
	\Bigg(
	\sum_i \max_j b_{ij}^{2p}
	\Bigg)^{1/2p}.
$$
The equivalence of the bounds in Theorem \ref{thm:moment} and
Corollary \ref{cor:moment} (up to the value of the universal constant)
follows immediately from this estimate.
\end{rem}

\subsection{Proof of Theorem \ref{thm:moment}}

Before we begin developing the proof of Theorem~\ref{thm:moment} in 
earnest, we make a simple observation: it suffices to prove the result
under the assumption that the diagonal entries of the matrix vanish.

\begin{lem}
\label{lem:nodiag}
Let $X$ be defined as in Theorem \ref{thm:moment}, let
$D_{ij}:=X_{ij}1_{i=j}$ be its diagonal part, and denote by
$\tilde X:=X-D$ its off-diagonal part. Then
$$
	(\mathbf{E}[\mathrm{Tr}[X^{2p}]])^{1/2p} \le
	(\mathbf{E}[\mathrm{Tr}[\tilde X^{2p}]])^{1/2p} +
	\sqrt{2p} \Bigg(\sum_i b_{ii}^{2p}\Bigg)^{1/2p}.
$$
\end{lem}

\begin{proof}
This follows immediately using the triangle inequality and the simple 
estimate $(\mathbf{E}[\mathrm{Tr}[D^{2p}]])^{1/2p}=
(\sum_i \mathbf{E}[X_{ii}^{2p}])^{1/2p} \le
\sqrt{2p} (\sum_i b_{ii}^{2p})^{1/2p}$.
\end{proof}

In view of Lemma \ref{lem:nodiag}, it suffices to prove the result of 
Theorem \ref{thm:moment} under the assumption that $b_{ii}=0$ for all $i$ 
(with constant $4$ rather than $5$ in the second term). While this 
observation is trivial, the elimination of the diagonal entries will be 
very helpful for the implementation of the combinatorial arguments that 
are used in the proof. \emph{We therefore assume in the rest of 
this section that $b_{ii}=0$ for all $i$.}

We now turn to the main part of the proof. Our starting point is the 
identity
$$
	\mathbf{E}[\mathrm{Tr}[X^{2p}]] =
	\sum_{u_1,\ldots,u_{2p}\in[n]}
	b_{u_1u_2}b_{u_2u_3}\cdots b_{u_{2p}u_1}
	\,\mathbf{E}[g_{u_1u_2}g_{u_2u_3}\cdots g_{u_{2p}u_1}].
$$
We view $u_1\to u_2\to\cdots\to u_{2p}\to u_1$ as a cycle in the complete 
undirected graph on $n$ points. By symmetry of the
Gaussian distribution, each distinct edge $\{u_k,u_{k+1}\}$ that appears 
in the cycle must be traversed an even number of times for a term in the 
sum to be nonzero. We will call cycles with this property \emph{even}.

Following \cite{BvH16}, the \emph{shape} $\mathbf{s}(\mathbf{u})$ of a
cycle $\mathbf{u}\in[n]^{2p}$ is defined by relabeling the vertices in
order of appearance. For example, the cycle $2\to 7\to 9\to 7\to 8\to 
7\to 2$ has shape $1\to 2\to 3\to 2\to 4\to 2\to 1$. We define
$$
	\mathcal{S}_{2p}:=
	\{\mathbf{s}(\mathbf{u}):\mathbf{u}\mbox{ is an even
	cycle of length }2p\}.
$$
We denote by $n_i(\mathbf{s})$ the number of distinct edges that
are traversed exactly $i$ times by a cycle of shape $\mathbf{s}$, and
by $m(\mathbf{s})$ the number of distinct vertices visited by $\mathbf{s}$.
In terms of these quantities, we can rewrite the moments
$\mathbf{E}[\mathrm{Tr}[X^{2p}]]$ as
$$
	\mathbf{E}[\mathrm{Tr}[X^{2p}]] =
	\sum_{\mathbf{s}\in\mathcal{S}_{2p}}
	\prod_{i\ge 1} \mathbf{E}[g^{2i}]^{n_{2i}(\mathbf{s})}
	\sum_{\mathbf{u}:\mathbf{s}(\mathbf{u})=\mathbf{s}}
	b_{u_1u_2}b_{u_2u_3}\cdots b_{u_{2p}u_1},
$$
where $g\sim N(0,1)$.
The key difficulty of the proof is to obtain the following estimate.

\begin{prop}
\label{prop:holder}
For any $\mathbf{s}\in\mathcal{S}_{2p}$
$$
	\sum_{\mathbf{u}:\mathbf{s}(\mathbf{u})=\mathbf{s}}
	b_{u_1u_2}b_{u_2u_3}\cdots b_{u_{2p}u_1}
	\le
	\Bigg(\sum_i\Bigg(\sum_j b_{ij}^2\Bigg)^p\Bigg)^{\frac{m(\mathbf{s})-1}{p}}
	\Bigg(\sum_i \max_j b_{ij}^{2p}\Bigg)^{1-\frac{m(\mathbf{s})-1}{p}}.
$$
\end{prop}

Proposition \ref{prop:holder} will be proved in section 
\ref{sec:holder} below. With this result in hand, we can 
readily complete the proof of Theorem \ref{thm:moment} as in \cite{BvH16}.

\begin{proof}[Proof of Theorem \ref{thm:moment}]
Define for simplicity
$$
	\sigma_p :=
	\Bigg(\sum_i\Bigg(\sum_j b_{ij}^2\Bigg)^p\Bigg)^{1/2p},
	\qquad
	\sigma_p^* := \Bigg(\sum_i \max_j b_{ij}^{2p}\Bigg)^{1/2p}.
$$
By rescaling the matrix $X$, we may assume without loss of generality
that $\sigma_p^*=1$. Then Proposition \ref{prop:holder} implies the estimate
$$
	\mathbf{E}[\mathrm{Tr}[X^{2p}]] \le
	\sum_{\mathbf{s}\in\mathcal{S}_{2p}}
	\sigma_p^{2(m(\mathbf{s})-1)}
	\prod_{i\ge 1} \mathbf{E}[g^{2i}]^{n_{2i}(\mathbf{s})}.
$$
On the other hand, let $Y$ be an $r\times r$ symmetric matrix
whose entries $Y_{ij}$ are i.i.d.\ standard Gaussian variables for $i\ge j$.
Then we have
\begin{align*}
	\mathbf{E}[\mathrm{Tr}[Y^{2p}]] &=
	\sum_{\mathbf{s}\in\mathcal{S}_{2p}}
	\card\{\mathbf{u}\in[r]^{2p}:\mathbf{s}(\mathbf{u})=\mathbf{s}\}
	\prod_{i\ge 1} \mathbf{E}[g^{2i}]^{n_{2i}(\mathbf{s})}
	\\
	&=\sum_{\mathbf{s}\in\mathcal{S}_{2p}}
	\frac{r!}{(r-m(\mathbf{s}))!}
	\prod_{i\ge 1} \mathbf{E}[g^{2i}]^{n_{2i}(\mathbf{s})}
\end{align*}
provided $p<r$ (which ensures that $r\ge m(\mathbf{s})$, as
$m(\mathbf{s})\le p+1$ for any even cycle).
As $(r-1)!/(r-m)!\ge (r-m+1)^{m-1}$ and using $m(\mathbf{s})\le p+1$,
we obtain
$$
	\mathbf{E}[\mathrm{Tr}[X^{2p}]] \le
	\frac{1}{r}\mathbf{E}[\mathrm{Tr}[Y^{2p}]]
	\le \mathbf{E}[\|Y\|^{2p}]
	\quad\mbox{for}\quad
	r=\lfloor\sigma_p^2\rfloor+p+1.
$$
We now invoke a standard bound on the norm of
Wigner matrices \cite[Lemma 2.2]{BvH16}
$$
	(\mathbf{E}[\mathrm{Tr}[X^{2p}]])^{1/2p} \le
	\mathbf{E}[\|Y\|^{2p}]^{1/2p} \le
	2\sqrt{r}+2\sqrt{2p} \le
	2\sigma_p + 4\sqrt{2p}
$$
to conclude the proof.
\end{proof}

\begin{rem}
Instead of Proposition \ref{prop:holder}, the following much simpler 
estimate 
$$
	\sum_{\mathbf{u}:\mathbf{s}(\mathbf{u})=\mathbf{s}}
	b_{u_1u_2}b_{u_2u_3}\cdots b_{u_{2p}u_1}
	\le
	n\,\Bigg(\max_i\sum_j b_{ij}^2\Bigg)^{m(\mathbf{s})-1}
	\Bigg(\max_{ij} b_{ij}^2\Bigg)^{p-m(\mathbf{s})+1}
$$
was obtained in \cite[Lemma 2.5]{BvH16}. Indeed, this estimate is an 
almost immediate consequence of the fact that every edge of an even cycle
is traversed at least twice. However, when combined with the rest of the 
argument, this estimate gives rise to a dimension-dependent bound on the 
Schatten norms
$$
	(\mathbf{E}[\mathrm{Tr}[X^{2p}]])^{1/2p} \lesssim
	n^{1/2p}\Bigg(
	\max_i\sqrt{\sum_j b_{ij}^2} + \max_{ij}b_{ij}\sqrt{p}\Bigg).
$$
The latter coincides with the optimal bound of Theorem \ref{thm:moment} 
only when $p\gtrsim\log n$. As we will presently see, the case $p\ll\log 
n$ is much more delicate, and the key feature of Proposition 
\ref{prop:holder} is that it gives the right dimension-free estimate.
\end{rem}

\subsection{Proof of Proposition \ref{prop:holder}}
\label{sec:holder}

To complete the proof of Theorem \ref{thm:moment}, it remains to prove the 
key estimate of Proposition \ref{prop:holder}. Its innocent-looking 
statement suggests that it should arise as an application of H\"older's 
inequality, and this is in essence the case. Nonetheless, we do not know a 
short proof of this fact. The difficulty in proving this result is that 
the manner in which H\"older's inequality must be applied depends 
nontrivially on the topology of the shape $\mathbf{s}$, and it is unclear 
at the outset how to identify the relevant structure. To facilitate the 
analysis, it will be convenient to place ourselves in a somewhat more 
general framework.

Let $\mathcal{G}_m$ be the set of all undirected, connected graphs 
$G=([m],E(G))$. For any edge $e=\{i,j\}$ of a graph $G\in\mathcal{G}_m$ 
and $\mathbf{v}=(v_i)_{i\in[m]}\in [n]^m$, we will write 
$v(e):=\{v_i,v_j\}$. The key quantity that we will investigate is
the following.

\begin{defn}
\label{defn:wg}
For any graph $G\in\mathcal{G}_m$, and any family $\mathbf{k}:=(k_e)_{e\in 
E(G)}$ of labelings of its edges by positive values $k_e>0$, we define
$$
	W^{\mathbf{k}}(G) := 
	\sum_{\mathbf{v}\in[n]^m}\prod_{e\in E(G)} b_{v(e)}^{k_e}.
$$
\end{defn}

Definition \ref{defn:wg} is more general than the quantities that appear
in Proposition \ref{prop:holder}. Indeed, given any shape 
$\mathbf{s}$ of length $2p$ with $m(\mathbf{s})=m$ distinct vertices, 
define a graph $G\in\mathcal{G}_m$ whose edges are given by
$E(G)=\{\{s_1,s_2\},\{s_2,s_3\},\ldots,\{s_{2p},s_1\}\}$, and let
$k_e$ be the number of times the edge $e\in E(G)$ is traversed by 
$\mathbf{s}$. Then clearly
$$
	\sum_{\mathbf{u}:\mathbf{s}(\mathbf{u})=\mathbf{s}}
	b_{u_1u_2}b_{u_2u_3}\cdots b_{u_{2p}u_1}
	=
	\sum_{v_1\ne v_2\ne\cdots\ne v_m}\prod_{e\in E(G)} b_{v(e)}^{k_e}
	\le W^{\mathbf{k}}(G).
$$
The conclusion of Proposition \ref{prop:holder} will therefore
follow immediately from the following more general statement about
the quantity $W^{\mathbf{k}}(G)$.

\begin{thm}
\label{thm:holder}
For any $G\in\mathcal{G}_m$ and $\mathbf{k}$ so that
$k_e\ge 2$ for every $e\in E(G)$, we have
$$
	W^{\mathbf{k}}(G) \le
	\Bigg(\sum_i\Bigg(\sum_j b_{ij}^2\Bigg)^{\frac{|\mathbf{k}|}{2}}
	\Bigg)^{\frac{2(m-1)}{|\mathbf{k}|}}
	\Bigg(\sum_i \max_j b_{ij}^{|\mathbf{k}|}\Bigg)^{1-
		\frac{2(m-1)}{|\mathbf{k}|}},
$$
where we denote $|\mathbf{k}|:=\sum_{e\in E(G)}k_e$.
\end{thm}

The remainder of this section is devoted to the proof of this result.

\subsubsection{Reduction to trees}

As a first step towards the proof of Theorem \ref{thm:holder}, we will 
show that it suffices to prove the result in the case where $G$ is a tree; 
the special structure of trees will be heavily exploited in the rest
of the proof.

In the sequel, we denote by $\mathcal{G}_m^{\rm tree}$ the set of trees 
on $m$ vertices, that is, the set of graphs $G\in\mathcal{G}_m$ such that 
$\card E(G)=m-1$. We will also denote by $G_I$ the subgraph induced by a 
graph $G\in\mathcal{G}_m$ on a subset $I\subseteq[m]$ of its vertices.

When $G\in\mathcal{G}_m^{\rm tree}$, the quantity $W^{\mathbf{k}}(G)$
takes a particularly simple form. Indeed, suppose without loss of 
generality that the vertices $[m]$ are ordered so that $m$ is a leaf of
$G$ and every $\ell<m$ is a leaf of the induced subtree $G_{[\ell]}$.
Then the tree is rooted at vertex $1$, and every vertex $\ell\ge 2$ has a 
unique parent vertex $i_\ell\le \ell-1$. In particular, then
$E(G)=\{\{1,2\},\{i_3,3\},\ldots,\{i_m,m\}\}$, so we can write
$$
	W^{\mathbf{k}}(G) =
	\sum_{v_1,\ldots,v_m\in[n]}
	b_{v_1v_2}^{k_2}b_{v_{i_3}v_3}^{k_3}\cdots
	b_{v_{i_m}v_m}^{k_m},
$$
where we denoted $k_{\{i_\ell,\ell\}} =: k_{\ell}$ for simplicity.
Conversely, for any $k_2,\ldots,k_m>0$ and
$i_\ell\le \ell-1$, the expression on the right-hand side arises as
$W^{\mathbf{k}}(G)$ for some $G\in\mathcal{G}_m^{\rm tree}$ (indeed,
one may generate any tree by starting at the root and repeatedly 
attaching a new vertex $\ell$ to a previously generated vertex
$i_\ell$).

We now show that among all graphs $G\in\mathcal{G}_m$, the value 
of $W^{\mathbf{k}}(G)$ is maximized by trees. This will allow us to 
restrict attention to trees in the rest of the proof.

\begin{lem}
\label{lem:tree}
For any $G\in\mathcal{G}_m$ and $\mathbf{k}=(k_e)_{e\in E(G)}$ with
$k_e>0$, there exist
$k_2',\ldots,k_m'>0$ with $\min_ik_i'\ge \min_ek_e$ and
$\sum_{i=2}^mk_i'=|\mathbf{k}|$ such that
$$
	W^{\mathbf{k}}(G) \le
	\max_{i_3,\ldots,i_m\in[m]:i_\ell\le \ell-1}
	\sum_{v_1,\ldots,v_m\in[n]}
        b_{v_1v_2}^{k_2'}b_{v_{i_3}v_3}^{k_3'}\cdots
        b_{v_{i_m}v_m}^{k_m'} \le
	\max_{G'\in\mathcal{G}_m^{\rm tree}}
	W^{\mathbf{k'}}(G').
$$
\end{lem}

\begin{proof}
Let $T$ be a spanning tree of $G$, and assume that the vertices $[m]$ of 
$G$ are ordered so that $m$ is a leaf of $T$ and every $\ell<m$ is a leaf 
of $T_{[\ell]}$ (this entails no loss of generality, as this can always be 
accomplished by relabeling the vertices of $G$). The only point of this 
assumption is that it ensures that the subgraph $G_{[\ell]}$ is connected 
for every $\ell$. We now define the numbers $k_2',\ldots,k_m'$ as
$$
	k_\ell' :=
	\sum_{e=\{i,\ell\}\in E(G_{[\ell]})} k_e.
$$
That is, $k_\ell'$ is the total weight of edges incident to $\ell$ in 
$G_{[\ell]}$. It is clear from this definition that 
$\min_\ell k_\ell'\ge\min_ek_e$ and $\sum_\ell k_\ell'=|\mathbf{k}|$.

The main part of the proof proceeds by induction. For the initial step, we 
begin by isolating the contribution of the vertex $m$ in the definition of 
$W^{\mathbf{k}}(G)$:
\begin{align*}
	W^{\mathbf{k}}(G) &=
	\sum_{\mathbf{v}\in[n]^m}
	\Bigg(
	\prod_{e\in E(G_{[m-1]})}
	b_{v(e)}^{k_e}
	\Bigg)\Bigg(
	\prod_{\{i,m\}\in E(G)} b_{v_iv_m}^{k(i)}
	\Bigg) \\
	&=
	\sum_{\mathbf{v}\in[n]^m}
	\prod_{\{i,m\}\in E(G)} 
	\Bigg(
	b_{v_iv_m}^{k_m'}
	\prod_{e\in E(G_{[m-1]})}
	b_{v(e)}^{k_e}
	\Bigg)^{k(i)/k_m'},
\end{align*}
where we denote $k_{\{i,m\}}=:k(i)$ and we observe that
$\sum_i k(i) = k_m'$. Therefore
\begin{align*}
	W^{\mathbf{k}}(G) &\le
	\prod_{\{i,m\}\in E(G)} 
	\Bigg(
	\sum_{\mathbf{v}\in[n]^m}
	b_{v_iv_m}^{k_m'}
	\prod_{e\in E(G_{[m-1]})}
	b_{v(e)}^{k_e}
	\Bigg)^{k(i)/k_m'}
	\\ &\le
	\max_{i_m\le m-1}
	\sum_{\mathbf{v}\in[n]^m}
	b_{v_{i_m}v_m}^{k_m'}
	\prod_{e\in E(G_{[m-1]})}
	b_{v(e)}^{k_e}
\end{align*}
by H\"older's inequality.

The argument for the inductive step proceeds along very similar lines. 
Suppose we have shown, for some $4\le r\le m$, the induction hypothesis
$$
	W^{\mathbf{k}}(G) \le
	\max_{i_r,\ldots,i_m:i_\ell\le \ell-1}
	\sum_{\mathbf{v}\in[n]^m}
	b_{v_{i_r}v_r}^{k_r'}\cdots
	b_{v_{i_m}v_m}^{k_m'}
	\prod_{e\in E(G_{[r-1]})}
	b_{v(e)}^{k_e}.
$$
By our assumption on the ordering of the vertices, $G_{[r-1]}$ is 
connected. In particular, this means that there exists at least one edge
between vertex $r-1$ and $[r-2]$. Isolating these edges
in the same manner as above yields
\begin{align*}
&	\sum_{\mathbf{v}\in[n]^m}
	b_{v_{i_r}v_r}^{k_r'}\cdots
	b_{v_{i_m}v_m}^{k_m'}
	\prod_{e\in E(G_{[r-1]})}
	b_{v(e)}^{k_e} = \mbox{}\\
&	\sum_{\mathbf{v}\in[n]^m}
	\prod_{\{i,r-1\}\in E(G_{[r-1]})}
\!\!
	\Bigg(
	b_{v_iv_{r-1}}^{k_{r-1}'}
	b_{v_{i_r}v_{r}}^{k_r'}\cdots
	b_{v_{i_m}v_{m}}^{k_m'}
	\prod_{e\in E(G_{[r-2]})}
	b_{v(e)}^{k_e}
	\Bigg)^{k(i)/k_{r-1}'},
\end{align*}
where we now redefined $k_{\{i,r-1\}}=:k(i)$. Applying
H\"older's inequality once more establishes the validity of the
induction hypothesis for $r\leftarrow r-1$.

The above induction guarantees the validity of the induction hypothesis
for $r=3$. This completes the proof, however, as $G_{[2]}$ contains the
single edge $\{1,2\}$.
\end{proof}

\subsubsection{Iterative pruning}

By virtue of Lemma \ref{lem:tree}, we have now reduced the proof of 
Theorem \ref{thm:holder} to the special case where the graph $G$ is a 
tree. To complete the proof, we will iteratively apply H\"older's 
inequality to the leaves of the tree. Unlike in the proof of Lemma 
\ref{lem:tree}, however, it is important in the present case to keep track 
of the powers of the different terms generated by H\"older's inequality, 
which introduces additional complications. To facilitate the requisite 
bookkeeping, it will be convenient to consider a further generalization of 
the quantity $W^{\mathbf{k}}(G)$.

In the following, let us suppose that each edge $e\in E(G)$ may be endowed 
with an independent (symmetric) weight matrix $b^{(e)}_{ij}$, and define
$$
	W(G) := 
	\sum_{\mathbf{v}\in[n]^m}
	\prod_{e\in E(G)} b_{v(e)}^{(e)}.
$$
We will recover $W(G)=W^{\mathbf{k}}(G)$ by setting $b^{(e)}_{ij}=b_{ij}^{k_e}$.

\begin{lem}
\label{lem:bizarrobl}
For any $G\in\mathcal{G}_m^{\rm tree}$ and $p_e\ge 1$ such that
$\sum_{e\in E(G)} 1/p_e=1$, we have
$$
	W(G) \le
	\prod_{e\in E(G)}
	\Bigg(
	\sum_i\Bigg(
	\sum_j b_{ij}^{(e)}\Bigg)^{p_e}
	\Bigg)^{1/p_e}.
$$
\end{lem}

This H\"older-type inequality is reminiscent of a special Brascamp-Lieb 
inequality (see, for example, \cite{BCCT10} and the references therein), 
but involving mixed $\ell_p(\ell_1)$ norms. We do not know whether it 
follows from a more general principle.

\begin{proof}
Throughout the proof we will assume without loss of generality that 
$m>2$, as the conclusion is trivial in the case $m=2$.

The proof again proceeds by induction. For the initial step, we begin by 
noting that any finite tree $G$ must have at least two leaves (that is,
vertices that have exactly one neighbor). Suppose that vertices 
$\ell,\ell'$ are leaves of $G$. Then
$$
        W(G) =
        \sum_{\mathbf{v}\in[n]^{I}}
	\Bigg(
	\sum_j b_{v_{i_\ell}j}^{(e_\ell)}
	\Bigg)
	\Bigg(
	\sum_j b_{v_{i_{\ell'}}j}^{(e_{\ell'})}
	\Bigg)
        \prod_{e\in E(G_{I})}b_{v(e)}^{(e)},
$$
where $I=[m]\backslash\{\ell,\ell'\}$ and $e_\ell=\{i_\ell,\ell\}$, 
$e_{\ell'}=\{i_{\ell'},\ell'\}$ are the unique edges connecting the leaves
$\ell,\ell'$ to $I$ (here we used that as $m>2$, the set $I$ is nonempty 
and $e_\ell\ne e_{\ell'}$). We can 
therefore estimate by H\"older's inequality
\begin{align*}
	W(G) \le\mbox{} &
	\Bigg[
        \sum_{\mathbf{v}\in[n]^{I}}
	\Bigg(
	\sum_j b_{v_{i_\ell}j}^{(e_\ell)}
	\Bigg)^{ 1 + \frac{p_{e_\ell}}{p_{e_{\ell'}}} }
        \prod_{e\in E(G_{I})}b_{v(e)}^{(e)}
	\Bigg]^{ \frac{ p_{e_{\ell'}} }{ p_{e_\ell}+p_{e_{\ell'}} } }
	\times\mbox{}\\
	&\Bigg[
        \sum_{\mathbf{v}\in[n]^{I}}
	\Bigg(
	\sum_j b_{v_{i_{\ell'}}j}^{(e_{\ell'})}
	\Bigg)^{ 1 + \frac{p_{e_{\ell'}}}{p_{e_{\ell}}} }
        \prod_{e\in E(G_{I})}b_{v(e)}^{(e)}
	\bigg]^{ \frac{ p_{e_{\ell}} }{ p_{e_\ell}+p_{e_{\ell'}} } }.
\end{align*}
Now observe the following properties of the right-hand side:
\begin{itemize}[label=\textbullet, leftmargin=*]
\item The graph $G_{I}$ is again a tree, as we remove only leaves
from $G$.
\item Both sides of the inequality are $1$-homogeneous in all the 
variables $b^{(e)}$.
\end{itemize}
The latter two properties will form the basis for the induction.

Let us now describe the induction step, which is again very similar.
Suppose we have shown, for some $r>1$, the induction hypothesis
$$
	W(G) \le
	\prod_{s=1}^S
	\Bigg[
        \sum_{\mathbf{v}\in[n]^{I_s}}
	\Bigg(
	\sum_j b_{v_{i_s}j}^{(e_s)}
	\Bigg)^{q_s}
        \prod_{e\in E(G_{I_s})}b_{v(e)}^{(e)}
	\Bigg]^{\frac{1}{\alpha_s}},
$$
where $S<\infty$, $I_s\subseteq[m]$, $i_s\in I_s$, $e_s\in E(G)
\backslash E(G_{I_s})$, and $\alpha_s,q_s\ge 1$ satisfy:
\begin{itemize}[label=\textbullet, leftmargin=*]
\item $\card I_s=r$ and $G_{I_s}$ is a tree for every $s$.
\item $q_s = \sum_{e\in E(G)\backslash E(G_{I_s})}p_{e_s}/p_e$.
\item The right-hand side is $1$-homogeneous
in all the variables $b^{(e)}$, $e\in E(G)$.
\end{itemize}
We aim to show that the induction hypothesis remains valid for
$r\leftarrow r-1$.

Consider a single term $s$ on the right-hand side.
As $G_{I_s}$ is a tree, it must have at least two leaves; in
particular, there is a vertex $\ell\ne i_s$ that is a leaf of
$G_{I_s}$. Thus
\begin{align*}
        &\sum_{\mathbf{v}\in[n]^{I_s}}
	\Bigg(
	\sum_j b_{v_{i_s}j}^{(e_s)}
	\Bigg)^{q_s}
        \prod_{e\in E(G_{I_s})}b_{v(e)}^{(e)}
	= \\ &\qquad\quad
        \sum_{\mathbf{v}\in[n]^{I_s'}}
	\Bigg(
	\sum_j b_{v_{i_s}j}^{(e_s)}
	\Bigg)^{q_s}
	\Bigg(
	\sum_j b_{v_{i'}j}^{(e')}
	\Bigg)
        \prod_{e\in E(G_{I_s'})}b_{v(e)}^{(e)},
\end{align*}
where $I_s'=I_s\backslash\{\ell\}$ and $i'\in I_s'$ is the unique vertex 
such that $e'=\{i',\ell\}\in E(G_{I_s})$. Applying
H\"older's inequality readily yields
\begin{align*}
        \sum_{\mathbf{v}\in[n]^{I_s}}
	\Bigg(
	\sum_j b_{v_{i_s}j}^{(e_s)}
	\Bigg)^{q_s}
        \prod_{e\in E(G_{I_s})}b_{v(e)}^{(e)}
	\le \mbox{} & 
	\Bigg[
        \sum_{\mathbf{v}\in[n]^{I_s'}}
	\Bigg(
	\sum_j b_{v_{i_s}j}^{(e_s)}
	\Bigg)^{q_s'}
        \prod_{e\in E(G_{I_s'})}b_{v(e)}^{(e)}
	\Bigg]^{\frac{q_s}{q_s'}}
	\\
	& \times \Bigg[
        \sum_{\mathbf{v}\in[n]^{I_s'}}
	\Bigg(
	\sum_j b_{v_{i'}j}^{(e')}
	\Bigg)^{q'}
        \prod_{e\in E(G_{I_s'})}b_{v(e)}^{(e)}
	\Bigg]^{\frac{1}{q'}},
\end{align*}
where $q_s'=\sum_{e\in E(G)\backslash E(G_{I_s'})}p_{e_s}/p_e$ and
$q'=\sum_{e\in E(G)\backslash E(G_{I_s'})}p_{e'}/p_e$ (here we used 
$q_s/q_s' + 1/q' = 1$).
Observe in particular that by construction, both sides in this inequality 
have the same degree of homogeneity in each variable $b^{(e)}$.

We have now shown how to bound a single term $s$ in the induction 
hypothesis. To conclude the induction argument, we replace every term in 
the induction hypothesis by the upper bound obtained by this procedure. We 
claim that the resulting bound again satisfies the induction hypothesis 
with $r\leftarrow r-1$, concluding the induction step. Indeed, by 
construction, each set $I_s'$ that appears in the new bound satisfies 
$\card I_s'=r-1$ and $G_{I_s'}$ is a tree (it was obtained from $G_{I_s}$
by removing a leaf). Moreover, by construction each term has the correct 
power $q_s'$. Finally, as each term in the induction hypothesis has been 
replaced by a term with the same homogeneity in each variable $b^{(e)}$, 
it follows immediately that the new bound is still $1$-homogeneous. This 
concludes the proof of the induction step.

The above induction guarantees the validity of the induction hypothesis
for $r=1$, that is, we have proved the following bound:
$$
	W(G) \le
	\prod_{s=1}^S
	\Bigg[
        \sum_{i}
	\Bigg(
	\sum_j b_{ij}^{(e_s)}
	\Bigg)^{p_{e_s}}
	\Bigg]^{\frac{1}{\alpha_s}} =
	\prod_{e\in E(G)}
	\Bigg[
        \sum_{i}
	\Bigg(
	\sum_j b_{ij}^{(e)}
	\Bigg)^{p_e}
	\Bigg]^{\frac{1}{\alpha_e}},
$$
where we defined $1/\alpha_e = \sum_{s:e_s=e}1/\alpha_{e_s}$. Moreover, 
the induction argument guarantees that the right-hand side is 
$1$-homogeneous in all variables $b^{(e)}$, $e\in E(G)$. It must therefore 
necessarily be the case that $\alpha_e = p_e$, concluding the proof. 
\end{proof}

Combining Lemmas \ref{lem:tree} and \ref{lem:bizarrobl} with the
assumption of Theorem \ref{thm:holder} yields:

\begin{cor}
\label{cor:preholder}
For any $G\in\mathcal{G}_m$ and $\mathbf{k}=(k_e)_{e\in E(G)}$ such that
$k_e\ge 2$ for every $e\in E(G)$, there exist 
$k_2',\ldots,k_m'\ge 2$ such that $\sum_{\ell=2}^mk_\ell'=|\mathbf{k}|$ and
$$
	W^{\mathbf{k}}(G) \le
	\prod_{\ell=2}^m
	\Bigg(
	\sum_i\Bigg(
	\sum_j b_{ij}^{k_\ell'}\Bigg)^{\frac{|\mathbf{k}|}{k_\ell'}}
	\Bigg)^{\frac{k_\ell'}{|\mathbf{k}|}}.
$$
\end{cor}

We can now conclude the proof of Theorem \ref{thm:holder}.

\begin{proof}[Proof of Theorem \ref{thm:holder}]
For any $2\le k\le K$, note that
\begin{align*}
	\sum_i\Bigg(\sum_j b_{ij}^k\Bigg)^{K/k} &\le
	\sum_i
	\Bigg(\sum_j b_{ij}^2\Bigg)^{K/k}
	\max_jb_{ij}^{K(k-2)/k}
	\\ &\le
	\Bigg(\sum_i\Bigg(\sum_j b_{ij}^2\Bigg)^{K/2}\Bigg)^{2/k}
	\Bigg(\sum_i\max_jb_{ij}^K\Bigg)^{(k-2)/k},
\end{align*}
where the second inequality used H\"older. The conclusion of Theorem
\ref{thm:holder} follows by applying this estimate to every term of
Corollary \ref{cor:preholder} (with $k=k_\ell'$, $K=|\mathbf{k}|$).
\end{proof}

\subsection{Proof of Corollary \ref{cor:moment}}
\label{sec:complex}

There are two separate issues that must be addressed in the proof of 
Corollary \ref{cor:moment}. The first is to prove a lower bound on 
$\mathbf{E}\|X\|_{S_p}^p$, which is elementary. The second is to extend 
the upper bound of Theorem \ref{thm:moment} to non-integer values of $p$,
which will be accomplished by complex interpolation.

We begin with the lower bound. We will need the following 
deterministic fact.

\begin{lem}
\label{lem:schlp}
For any (not necessarily symmetric) matrix $M$ and
$2\le p\le\infty$
$$
	\|M\|_{S_p} \ge
	\|M\|_{\ell_p(\ell_2)} :=
	\Bigg(\sum_i\Bigg(\sum_j M_{ij}^2\Bigg)^{p/2}\Bigg)^{1/p}.
$$
\end{lem}

\begin{proof}
Note that as $p\ge 2$, we can write
\begin{align*}
	\|M\|_{S_p}^2 =
	\|MM^*\|_{S_{p/2}} &=
	\sup_{\|Z\|_{S_{p/(p-2)}}\le 1} \mathrm{Tr}[ZMM^*]
	\\ &\ge \sup_{\|v\|_{p/(p-2)}\le 1} 
	\mathrm{Tr}[\mathrm{diag}(v)MM^*] = 
	\Bigg(\sum_i (MM^*)_{ii}^{p/2}\Bigg)^{2/p}.
\end{align*}
The result follows readily.
\end{proof}

\begin{proof}[Proof of Corollary \ref{cor:moment}: lower bound]
We compute two distinct lower bounds. First, note that by Lemma 
\ref{lem:schlp} and Jensen's inequality,
$$
	(\mathbf{E}\|X\|_{S_p}^p)^{1/p} \ge
	\Bigg(\sum_i\Bigg(\sum_j 
	\mathbf{E}X_{ij}^2\Bigg)^{p/2}\Bigg)^{1/p} =
	\Bigg(\sum_i\Bigg(\sum_j 
	b_{ij}^2\Bigg)^{p/2}\Bigg)^{1/p}.
$$
On the other hand, using again Lemma \ref{lem:schlp}, we can estimate
$$
	(\mathbf{E}\|X\|_{S_p}^p)^{1/p} \ge
	\Bigg(\sum_{i,j}\mathbf{E}|X_{ij}|^p\Bigg)^{1/p}
	\gtrsim \sqrt{p}\Bigg(\sum_{i,j}b_{ij}^p\Bigg)^{1/p},
$$
where we used $(\sum_j X_{ij}^2)^{p/2} \ge \sum_j |X_{ij}|^p$ and
$\mathbb{E}[|g|^p]^{1/p} \asymp \sqrt{p}$ when $g$ is standard Gaussian.
Averaging these two bounds concludes the proof of the lower bound.
\end{proof}

In the rest of this section, we will use standard facts and definitions 
from the theory of complex interpolation that can be found, for example, 
in \cite[Chapter 8]{Pis16}.

In order to apply complex interpolation, it will be most convenient to 
work with non-symmetric matrices rather than symmetric ones. Our basic 
object of study will be defined as follows. Let $(\tilde 
g_{ij})_{i,j\in[n]}$ be i.i.d.\ (real) standard Gaussian variables, and 
define the linear mapping $T:\mathbb{C}^{n\times n}\to 
\bigcap_p L^p(\Omega;\mathbb{C}^{n\times n})$ as
$$
	T((a_{ij})_{i,j\in[n]}) = (a_{ij}\tilde g_{ij})_{i,j\in [n]}.
$$
That is, $T$ maps the complex coefficients $(a_{ij})$ to the non-symmetric 
complex random matrix with entries $(a_{ij}\tilde g_{ij})$.
From Theorem \ref{thm:moment}, we deduce the following (here we define the 
random matrix norm
$\|X\|_{L^{p}(S_{p})} := (\mathbf{E}\|X\|_{S_{p}}^{p})^{1/p}$).

\begin{lem}
\label{lem:complext}
We have for all $p\in\mathbb{N}$
$$
	\|T((a_{ij}))\|_{L^{2p}(S_{2p})}\lesssim
	\|(a_{ij})\|_{\ell_{2p}(\ell_2)} +
	\|(a_{ji})\|_{\ell_{2p}(\ell_2)} + \sqrt{p}\,\|(a_{ij})\|_{\ell_{2p}}.
$$
\end{lem}

\begin{proof}
By writing $T((a_{ij})) = (\mathop{\mathrm{Re}}a_{ij}\tilde g_{ij}) +
i(\mathop{\mathrm{Im}}a_{ij}\tilde g_{ij})$ and applying the triangle 
inequality, it evidently suffices to prove the claim for the case where
all $a_{ij}$ are real. To this end, form the $2n\times 2n$ symmetric 
matrix
$$
	X =
	\begin{pmatrix}
	0 & (a_{ij}\tilde g_{ij}) \\
	(a_{ji}\tilde g_{ji}) & 0
	\end{pmatrix},
$$
and note that $\|X\|_{S_{2p}}^{2p} =
\|X^2\|_{S_p}^p = 2\|(a_{ij}\tilde g_{ij})\|_{S_{2p}}^{2p}$.
The conclusion now follows readily by applying Theorem \ref{thm:moment}
to the real symmetric random matrix $X$.
\end{proof}

Observe that all three norms that appear on the right-hand side of Lemma 
\ref{lem:complext} are Banach lattice norms, as they are monotone in 
$|a_{ij}|$ for every $i,j$. This enables us to apply a very convenient 
observation of \cite[Theorem 2]{Mal87}: intersections of Banach lattices 
are well-behaved under complex interpolation, in the sense that if
$B_i=(\mathbb{C}^m,\|\cdot\|_{B_i})$, $i=0,1,2$ are finite-dimensional 
Banach lattices, then $\|x\|_{(B_0,B_1\cap B_2)_\theta}\le
2\|x\|_{(B_0,B_1)_\theta\cap(B_0,B_2)_\theta}$, where
$\|x\|_{B\cap C}:=\max(\|x\|_B,\|x\|_C)$. It remains to combine this 
observation with standard interpolation arguments.

\begin{proof}[Proof of Corollary \ref{cor:moment}: upper bound]
Let $1\le p<\infty$ be arbitrary. Define $q:=\lceil p\rceil$, and
$\theta\in(0,1)$ by $(1-\theta)/2 + \theta/2q = 1/2p$. By Lemma 
\ref{lem:complext}, we have
\begin{align*}
&	\|T((a_{ij}))\|_{L^{2}(S_{2})} \lesssim
	\|(a_{ij})\|_{\ell_2}, \\
&	\|T((a_{ij}))\|_{L^{2q}(S_{2q})} \lesssim
	\|(a_{ij})\|_{\ell_{2q}(\ell_2)} +
	\|(a_{ji})\|_{\ell_{2q}(\ell_2)} + \sqrt{p}\,\|(a_{ij})\|_{\ell_{2q}}.
\end{align*}
We now recall that
$$
	(\ell_2,\ell_{2q}(\ell_2))_\theta= \ell_{2p}(\ell_2),\qquad
	(\ell_2,\ell_{2q})_\theta=\ell_{2p},\qquad
	(L^2(S_2),L^{2q}(S_{2q}))_\theta=L^{2p}(S_{2p})
$$
isometrically, see \cite[Theorem 8.21 and (14.3)]{Pis16}. Thus the 
above-mentioned lattice property \cite{Mal87} and the fundamental
theorem of interpolation \cite[Theorem 8.8]{Pis16} yield
$$
	\|T((a_{ij}))\|_{L^{2p}(S_{2p})} \lesssim
	\|(a_{ij})\|_{\ell_{2p}(\ell_2)} +
	\|(a_{ji})\|_{\ell_{2p}(\ell_2)} + \sqrt{p}\,\|(a_{ij})\|_{\ell_{2p}}.
$$
That is, we have shown that Lemma \ref{lem:complext} extends to all
(non-integer) $1\le p<\infty$.

Finally, let the symmetric matrix $X$ be as in Theorem \ref{thm:moment}.
To deduce the conclusion of Corollary \ref{cor:moment}, note that we can 
estimate by the triangle inequality
$$
	\|X\|_{L^p(S_p)} \le
	\|T((b_{ij}1_{i\ge j}))\|_{L^p(S_p)} +
	\|T((b_{ij}1_{i<j}))\|_{L^p(S_p)},
$$
where we used that the entries above (or below) the diagonal of $X$ are 
independent. The conclusion now follows readily from the above estimate
on $T$.
\end{proof}

\section{Norm estimates}
\label{sec:norm}

The main result of this section is the following theorem. When combined 
with Corollary \ref{cor:moment} and Remark \ref{rem:strange}, this 
concludes the proof of Theorem \ref{thm:main}.

\begin{thm}
\label{thm:schatten}
Let $X$ be an $n\times n$ symmetric matrix with 
$X_{ij}=b_{ij}g_{ij}$, where $b_{ij}\ge 0$ and $g_{ij}$ are i.i.d.\ 
standard Gaussian variables for $i\ge j$. Then for $2\le p\le\infty$
\begin{align*}
	\mathbf{E}\|X\|_{S_p} &\asymp
	\mathbf{E}\Bigg[
	\Bigg(\sum_i\Bigg(\sum_j X_{ij}^2\Bigg)^{p/2}\Bigg)^{1/p}
	\Bigg] 
	\\
	&\asymp
	\Bigg(\sum_i\Bigg(\sum_j b_{ij}^2\Bigg)^{p/2}\Bigg)^{1/p} +
	\max_{i\le e^p}\max_j b_{ij}^*\sqrt{\log i} +
	\sqrt{p}\Bigg(
	\sum_{i\ge e^p}\max_j {b_{ij}^{*}}^p
	\Bigg)^{1/p},
\end{align*}
where the matrix $(b_{ij}^*)$ is obtained by permuting the rows and 
columns of the matrix $(b_{ij})$ such that
$\max_j b_{1j}^* \ge \max_j b_{2j}^* \ge \cdots \ge \max_j b_{nj}^*$.
\end{thm}

\begin{rem}
Theorem \ref{thm:schatten} is concerned with the regime $2\le p\le\infty$. 
That its conclusion remains valid in the regime $1\le p<2$ follows already 
from the much more general noncommutative Khintchine inequality 
\cite[section 6]{PX03}, when specialized to symmetric matrices with 
independent Gaussian entries (the formulation for non-symmetric matrices 
is a bit more subtle). In fact, the noncommutative Khintchine 
inequality is valid in the range $1\le p<\infty$, but becomes increasingly 
suboptimal for large $p$. The key novelty of Theorem \ref{thm:schatten} is 
that it captures the precise behavior for $p\to\infty$ in the setting
where the matrix has independent entries. Our result is already
qualitatively new for $p=\infty$; it implies, for example, the
characterization of bounded infinite-dimensional random matrices of
Corollary \ref{cor:inf} in the introduction. The proof of the latter 
result will be given at the end of this section.
\end{rem}

\begin{rem}
Theorem \ref{thm:schatten} proves that $\mathbf{E}\|X\|_{S_p}\asymp 
\mathbf{E}\|X\|_{\ell_p(\ell_2)}$. It is interesting to note that 
Corollary \ref{cor:moment} admits a similar interpretation: one may show 
that its conclusion can be rewritten as 
$(\mathbf{E}\|X\|_{S_p}^p)^{1/p}\asymp(\mathbf{E}\|X\|_{\ell_p(\ell_2)}^p)^{1/p}$. 
These observations suggest that perhaps other moments of the random 
variables $\|X\|_{S_p}$ and $\|X\|_{\ell_p(\ell_2)}$ may also be 
comparable. In fact, we will show that the distributions of these random 
variables are comparable in a much stronger sense, as was stated 
in Corollary \ref{cor:schattentail} in the introduction. This result
will also be proved at the end of this section.
\end{rem}

\begin{rem}
\label{rem:nonzeromean}
Throughout the paper, we focus attention on centered Gaussian 
matrices $X$. There is however no loss of generality in doing so:
a matrix with arbitrary mean can always be reduced to the centered case
using that
$$
        \mathbf{E}[\|A+X\|_{S_p}^q]^{1/q} \asymp \|A\|_{S_p} +
        \mathbf{E}[\|X\|_{S_p}^q]^{1/q}
$$
for any deterministic matrix $A$, centered random matrix $X$, and $q\ge 1$. 
Indeed, the upper bound is obvious by the triangle inequality. To show the
lower bound, note that $\mathbf{E}[\|A+X\|_{S_p}^q]^{1/q}\ge\|A\|_{S_p}$ 
by Jensen's inequality (as $\mathbf{E}X=0$), while
$\mathbf{E}[\|A+X\|_{S_p}^q]^{1/q}\ge\mathbf{E}[\|X\|_{S_p}^q]^{1/q}-
\|A\|_{S_p}$ by the reverse triangle inequality. Adding twice the first 
inequality to the second inequality gives a the desired lower bound.
\end{rem}

\subsection{The mixed norm}

Before we turn to the main part of the proof of 
Theorem~\ref{thm:schatten}, we first establish the explicit expression 
given there in terms of the coefficients $b_{ij}$. This explicit 
expression will play an important role in the subsequent analysis of the 
random matrix. In the special case $p=\infty$, this result was proved in 
\cite[Theorem~1.2]{vH17}; we extend it here to any value of $2\le p\le 
\infty$.

We will need the following elementary result.

\begin{lem}
\label{lem:gausslp}
Let $G=(G_1,\ldots,G_n)$ be independent with $G_i\sim N(0,\sigma_i^2)$.
Then
$$
	\mathbf{E}\|G\|_{\ell_p} \asymp
	\max_{i\le e^p}\sigma_i^*\sqrt{\log(i+1)} +
	\sqrt{p}\,\Bigg(\sum_{i\ge e^p} {\sigma_i^*}^p\Bigg)^{1/p}
$$
for $2\le p\le \infty$, where $(\sigma_i^*)$ is the nonincreasing 
rearrangement of $(\sigma_i)$. The upper bound remains valid without 
independence and when $G_i$ is only $\sigma_i$-subgaussian.
\end{lem}

\begin{proof}
By permutation invariance, we can assume in the following 
without loss of generality that $\sigma_i$ are positive and nonincreasing 
(so that $\sigma_i=\sigma_i^*$).

Let us begin with the upper bound. By the triangle inequality, we have
$\mathbf{E}\|G\|_{\ell_p}\le
\mathbf{E}\|G_{\le e^p}\|_{\ell_p} + \mathbf{E}\|G_{\ge e^p}\|_{\ell_p}$,
where $G_{\le k}:=(G_1,\ldots,G_{\lfloor k\rfloor})$, 
$G_{\ge k}:=(G_{\lceil k\rceil },\ldots,G_n)$.
Now recall that for a vector $x\in\mathbb{R}^k$, we have
$\|x\|_{\ell_p}\le k^{1/p}\|x\|_{\ell_\infty}$. Therefore
$$
	\mathbf{E}\|G_{\le e^p}\|_{\ell_p} \le
	e\,\mathbf{E}\|G_{\le e^p}\|_{\ell_\infty}
	\lesssim 
	\max_{i\le e^p}\sigma_i\sqrt{\log(i+1)},
$$
where the last inequality can be found in \cite[Lemma 2.3]{vH17}. On the 
other hand,
$$
	\mathbf{E}\|G_{\ge e^p}\|_{\ell_p} \le
	(\mathbf{E}\|G_{\ge e^p}\|_{\ell_p}^p)^{1/p} \lesssim
	\sqrt{p}\,\Bigg(\sum_{i\ge e^p} \sigma_i^p\Bigg)^{1/p},
$$
where we used that $\mathbf{E}[|G_i|^p]^{1/p}\lesssim \sigma_i\sqrt{p}$
when $G_i$ is $\sigma_i$-subgaussian. This concludes the proof of the 
upper bound. Note that the only assumption that was used so far is
that each $G_i$ is $\sigma_i$-subgaussian; no independence was assumed.

We now turn to the lower bound. In the sequel, we will make use of the 
stronger assumptions that $G_i\sim N(0,\sigma_i^2)$ and that $(G_i)$ are 
independent. First, note that
$$
	\max_{i\le e^p}\sigma_i\sqrt{\log(i+1)} \asymp
	\mathbf{E}\|G_{\le e^p}\|_{\ell_\infty} \le
	\mathbf{E}\|G\|_{\ell_p},
$$
where the first inequality is given in \cite[Lemma 2.4]{vH17} and the 
second inequality is trivial. On the other hand, let us note that
$$
	\sqrt{p}\,\Bigg(\sum_{i\ge e^p} \sigma_i^p\Bigg)^{1/p}
	\asymp
	(\mathbf{E}\|G_{\ge e^p}\|_{\ell_p}^p)^{1/p} 
	\lesssim
	\mathbf{E}\|G_{\ge e^p}\|_{\ell_p} + 
	\max_{i\ge e^p}\sigma_i\sqrt{p},
$$
where the first inequality follows as $\mathbf{E}[|G_i|^p]^{1/p}\asymp
\sigma_i\sqrt{p}$ when $G_i\sim N(0,\sigma_i^2)$, and the second 
follows using the triangle inequality and that
$\|G_{\ge e^p}\|_{\ell_p}-\mathbf{E}\|G_{\ge e^p}\|_{\ell_p}$ is 
$\max_{i\ge e^p}\sigma_i$-subgaussian by Gaussian concentration 
\cite[Theorem 5.8]{BLM13}. But as we assumed that $\sigma_i$ are 
nonincreasing, we evidently have 
$$
	\max_{i\ge e^p}\sigma_i\sqrt{p} \le
	\sigma_{\lfloor e^p\rfloor} \sqrt{p} \le
	\max_{i\le e^p}\sigma_i\sqrt{\log(i+1)} \lesssim
	\mathbf{E}\|G_{\le e^p}\|_{\ell_p}.
$$
We can therefore easily conclude that
$$
	\sqrt{p}\,\Bigg(\sum_{i\ge e^p} \sigma_i^p\Bigg)^{1/p}
	\lesssim \mathbf{E}\|G\|_{\ell_p},
$$
and the proof is completed by averaging the two lower bounds on
$\mathbf{E}\|G\|_{\ell_p}$.
\end{proof}

We are now ready to prove the explicit bound given in Theorem 
\ref{thm:schatten}.

\begin{cor}
\label{cor:expl}
Under the assumptions of Theorem \ref{thm:schatten}, we have for all
$2\le p\le\infty$
\begin{align*}
	&\mathbf{E}\Bigg[
	\Bigg(\sum_i\Bigg(\sum_j X_{ij}^2\Bigg)^{p/2}\Bigg)^{1/p}
	\Bigg] \asymp \mbox{} 
	\\
	&\qquad
	\Bigg(\sum_i\Bigg(\sum_j b_{ij}^2\Bigg)^{p/2}\Bigg)^{1/p} +
	\max_{i\le e^p}\max_j b_{ij}^*\sqrt{\log i} +
	\sqrt{p}\Bigg(
	\sum_{i\ge e^p}\max_j {b_{ij}^{*}}^p
	\Bigg)^{1/p}.
\end{align*}
\end{cor}

\begin{proof}
Let us begin with the upper bound. Define the vector
$Z=(Z_1,\ldots,Z_n)$ with $Z_i := (\sum_j X_{ij}^2)^{1/2}$.
Then we can estimate using the triangle inequality
$$
	\mathbf{E}\Bigg[ 
        \Bigg(\sum_i\Bigg(\sum_j X_{ij}^2\Bigg)^{p/2}\Bigg)^{1/p}
        \Bigg] =
	\mathbf{E}\|Z\|_{\ell_p}
	\le
	\|\mathbf{E}Z\|_{\ell_p} +
	\mathbf{E}\|Z-\mathbf{E}Z\|_{\ell_p}.
$$
It follows easily from Jensen's inequality that
$$
	\|\mathbf{E}Z\|_{\ell_p}  \le
	\Bigg(\sum_i\Bigg(\sum_j b_{ij}^2\Bigg)^{p/2}\Bigg)^{1/p}.
$$
On the other hand, by Gaussian concentration \cite[Theorem
5.8]{BLM13}, the random 
variable $Z_i-\mathbf{E}Z_i$ is $\max_j b_{ij}$-subgaussian. We 
therefore obtain
$$
	\mathbf{E}\|Z-\mathbf{E}Z\|_{\ell_p} \lesssim
	\max_{i\le e^p}\max_j b_{ij}^*\sqrt{\log(i+1)} +
	\sqrt{p}\Bigg(
	\sum_{i\ge e^p}\max_j {b_{ij}^{*}}^p
	\Bigg)^{1/p}
$$
by Lemma \ref{lem:gausslp}, which completes the proof of the upper bound. 
(In the final bound, we estimated $\log(i+1)\le 1+\log i$ for aesthetic 
reasons; this does not entail any loss provided we slightly increase
the constant in front of the first two terms.)

To prove the lower bound, denote by $k_i$ the entry of the $i$th row of 
the matrix that has the largest variance, that is, $b_{ik_i}=\max_jb_{ij}$.
Then we have
\begin{align*}
	\mathbf{E}\Bigg[ 
        \Bigg(\sum_i\Bigg(\sum_j X_{ij}^2\Bigg)^{p/2}\Bigg)^{1/p}
        \Bigg] &\ge
	\mathbf{E}\Bigg[ 
        \Bigg(\sum_i|X_{ik_i}|^p\Bigg)^{1/p}
        \Bigg] \\ &\gtrsim
        \max_{i\le e^p}\max_j b_{ij}^*\sqrt{\log i} +
        \sqrt{p}\Bigg(
        \sum_{i\ge e^p}\max_j {b_{ij}^{*}}^p
        \Bigg)^{1/p}
\end{align*}
using Lemma \ref{lem:gausslp}. To be precise, note that the random 
variables $X_{ik_i}$ and $X_{jk_j}$ are either independent or identically 
equal, the latter happening if $k_i=j$ and $k_j=i$. However, as each 
independent variable appears at most twice in the vector $(X_{ik_i})$, it 
is readily verified that the conclusion of Lemma \ref{lem:gausslp} remains 
valid in this setting modulo a suitable modification of the universal 
constants.

On the other hand, we can lower bound by Jensen's inequality
$$
	\mathbf{E}\Bigg[ 
        \Bigg(\sum_i\Bigg(\sum_j X_{ij}^2\Bigg)^{p/2}\Bigg)^{1/p}
        \Bigg] =
	\mathbf{E}\|Z\|_{\ell_p}
	\ge
	\|\mathbf{E}Z\|_{\ell_p}.
$$
Now note that by the Gaussian Poincar\'e inequality \cite[Theorem 
3.20]{BLM13}, we have
$$
	\sum_j b_{ij}^2 = \mathbf{E}[Z_i^2] =
	\mathbf{E}[Z_i]^2 + \mathrm{Var}(Z_i) \le
	\mathbf{E}[Z_i]^2  +\max_{j}b_{ij}^2 \lesssim
	\mathbf{E}[Z_i]^2.
$$
We therefore obtain
$$
	\mathbf{E}\Bigg[ 
        \Bigg(\sum_i\Bigg(\sum_j X_{ij}^2\Bigg)^{p/2}\Bigg)^{1/p}
        \Bigg] 
	\gtrsim
	\mathbf{E}\Bigg[ 
        \Bigg(\sum_i\Bigg(\sum_j b_{ij}^2\Bigg)^{p/2}\Bigg)^{1/p}
        \Bigg],
$$
and the proof is concluded by averaging the two lower bounds.
\end{proof}

\subsection{The operator norm}
\label{sec:opnorm}

The aim of this section is to prove Theorem \ref{thm:schatten} in the case 
$p=\infty$. In fact, this turns out to be the most interesting case: in 
the next section, we will see that the proof of Theorem \ref{thm:schatten} 
for arbitrary $2\le p\le\infty$ follows rather quickly by combining the 
case $p=\infty$ with Theorem \ref{thm:moment}. In the following, we will 
denote the operator norm as $\|X\|:=\|X\|_{S_\infty}$ for notational 
simplicity.

Before we describe the main construction behind the proof, we must first 
recall two suboptimal bounds on $\mathbf{E}\|X\|$. First, we observe that 
a useful bound can already be deduced from Theorem \ref{thm:moment}; for 
the operator norm, this result was obtained (by a significantly simpler 
variant of the proof of Theorem \ref{thm:moment}) in \cite{BvH16}.

\begin{thm}[\cite{BvH16}, Theorem 1.1]
\label{thm:bvh}
Let $X$ be an $n\times n$ symmetric matrix with 
$X_{ij}=b_{ij}g_{ij}$, where $b_{ij}\ge 0$ and $g_{ij}$ are i.i.d.\ 
standard Gaussian variables for $i\ge j$. Then
$$
	\mathbf{E}\|X\|\lesssim
	\max_i\sqrt{\sum_j b_{ij}^2} + 
	\max_{ij}b_{ij}\sqrt{\log n}.
$$
\end{thm}

\begin{proof}
As $\mathbf{E}\|X\|\le(\mathbf{E}\|X\|_{S_{2p}}^{2p})^{1/2p}$,
we can apply Theorem \ref{thm:moment} for any $p$. We conclude by 
observing that $\|x\|_{\ell_p}\asymp \|x\|_{\ell_\infty}$ for 
$x\in\mathbb{R}^n$ if we choose $p\sim\log n$.
\end{proof}

The problem with this bound is that it is dimension-dependent, while the 
sharp result of Theorem \ref{thm:schatten} is inherently dimension-free. 
Unfortunately, as is indicated by Corollary \ref{cor:moment}, any bound on 
the $p$th moment with $p\sim\log n$ must necessarily depend on the 
dimension $n$. One therefore cannot hope to obtain a dimension-free bound 
by an improvement of the moment method. Instead, an entirely different 
approach was introduced in \cite{vH17} to obtain a dimension-free bound on 
$\mathbf{E}\|X\|$ through the theory of Gaussian processes. We recall the 
following result without proof.

\begin{thm}[\cite{vH17}, paragraph after Theorem 1.4]
\label{thm:vh}
In the setting of Theorem \ref{thm:bvh}
$$
	\mathbf{E}\|X\|\lesssim
	\max_i\sqrt{\sum_j b_{ij}^2} + 
	\max_{ij}b_{ij}^*\log i.
$$
\end{thm}

The advantage of Theorem \ref{thm:vh} is that it is dimension-free, in a 
manner strongly reminiscent of the sharp result of Theorem 
\ref{thm:schatten}. However, the result is suboptimal in a different 
sense, as its second term is too large ($\log i$ rather than $\sqrt{\log i}$).

We will presently show that the sharp bound of Theorem \ref{thm:schatten} 
in the case $p=\infty$ can be obtained by efficiently exploiting the two 
suboptimal bounds of Theorems~\ref{thm:bvh} and \ref{thm:vh}. This will be 
achieved by decomposing the matrix $X$ into different pieces: dominant 
pieces of small dimension, which are controlled optimally by the 
dimension-dependent bound, and a small remainder of large dimension, which 
can be controlled by the dimension-free bound. This decomposition not only 
allows us to conclude the proof, but also provides significant insight 
into the structure of large random matrices with bounded operator norm.

We now proceed to develop the details of the construction. Fix a matrix 
$X$ as in Theorem \ref{thm:schatten}, and define for the remainder of this 
subsection the quantities
$$
	a := 
        \max_i\sqrt{\sum_j b_{ij}^2},\qquad\qquad
	b :=
        \max_{ij}b_{ij}^*\sqrt{\log i}.
$$
Our aim is to prove the upper bound of Theorem \ref{thm:schatten}
in the case $p=\infty$, that is:

\begin{thm}
\label{thm:latala}
$\mathbf{E}\|X\|\lesssim a+b$.
\end{thm}

Theorem \ref{thm:latala} completes the proof of Theorem \ref{thm:schatten} 
in the case $p=\infty$, as the corresponding lower bound follows trivially 
from Lemma \ref{lem:schlp} and Corollary \ref{cor:expl}.

At the heart of the proof lies the observation that the quantities $a$ and 
$b$ provide different types of control on the coefficients $b_{ij}$.
On the one hand, by definition, 
\begin{equation}
\tag{B}\label{B}
	\card\bigg\{i:\max_j b_{ij} > \frac{b}{\sqrt{\log k}}\bigg\}
	< k
\end{equation}
for every $k$. On the other hand, for every given $j$, 
\begin{equation}
\tag{A}\label{A}
	\card\bigg\{i:b_{ij} > \frac{a}{\sqrt{k}}\bigg\}< k
\end{equation}
for every $k$. In other words, the maximal entry across columns 
$\max_j b_{ij}$ decays, when rearranged in decreasing order, as 
$b/\!\sqrt{\log i}$; while inside each given column $j$, the entries 
$b_{ij}$ decay, when rearranged in decreasing order, as $a/\!\sqrt{i}$.
Of course, the ordering of entries in each column is different, so we 
cannot simultaneously rearrange all columns in decreasing order. 
Instead, we construct one rearrangement that benefits from both 
properties by alternating between \eqref{B} and \eqref{A}.

We now describe the construction. We choose a permutation 
$(i_1,\ldots,i_n)$ of the indices $\{1,\ldots,n\}$ by a simple 
algorithm. After $k$ steps of the algorithm, we will have selected exactly 
$N_k:=2^{2^k}$ indices which we denote by $I_k:=\{i_1,\ldots,i_{N_k}\}$.
\vskip.1cm
\begin{description}[leftmargin=2.4cm,style=nextline]
\itemsep\abovedisplayskip
\item[\bf Step 1] Choose $N_1$ indices $i$ for which the
quantity $\max_jb_{ij}$ is largest.
\item[\bf Step k.\ (a)] Among the remaining indices, 
choose $N_{k-1}N_{k-2}$ indices $i$ that contain the
$N_{k-2}$ largest entries $b_{ij}$ of each column $j\in I_{k-1}$.
\item[\bf Step k.\ (b)] Among the remaining indices, choose
$N_k-N_{k-1}-N_{k-1}N_{k-2}$ indices $i$ for which the
quantity $\max_jb_{ij}$ is largest.
\end{description}
\vskip.1cm
Note that this iteration is well defined as $N_k-N_{k-1}-N_{k-1}N_{k-2}>1$
for all $k\ge 2$, and as $|I_k|=N_k$ for each $k$ by construction. The 
algorithm terminates when we have selected all $n$ indices. In the
remainder of the proof, we will always work with the rearrangement of the
indices constructed by this algorithm. We will therefore assume from
now on, without loss of generality, that the algorithm selects the 
identity permutation $i_k=k$ (otherwise we may permute the rows and 
columns of the matrix such that this is the case, which does not alter
any of the quantities of interest; this is just a relabeling of the 
indices of the matrix).

The construction we have just given ensures that we can control the 
magnitudes of the entries in different parts of the matrix.

\begin{lem}
\label{lem:core}
After rearranging as above, the following hold for all $k\ge 1$:
\begin{enumerate}[leftmargin=*,label=\roman*.]
\item $b_{ij}\lesssim b2^{-k/2}$ when $i\ge N_k$.
\item $b_{ij}\le a2^{-2^{k-2}}$ when $j\le N_k$ and
$i\ge N_k+N_kN_{k-1}$.
\end{enumerate}
By symmetry, the identical bounds hold for $b_{ji}$.
\end{lem}

\begin{proof}
After $k$ steps of the algorithm, we have selected at least
$$
	N_1 + \sum_{s=2}^k (N_s-N_{s-1}-N_{s-1}N_{s-2}) \ge
	N_{k-1}
$$
indices $i$ for which the quantity $\max_j b_{ij}$ is largest. Therefore,
by \eqref{B}, we have $b_{ij}\le b/\!\sqrt{\log N_{k-1}}$ whenever
$i\ge N_k$. This proves part $i$.
On the other hand, after $k+1$ steps of the algorithm, the $N_{k-1}$ 
largest entries $b_{ij}$ of each of the first $N_k$ columns $j$ are 
contained in the first $N_k+N_kN_{k-1}$ rows $i$. Therefore, when
$j\le N_k$ and $i\ge N_k+N_kN_{k-1}$, we have $b_{ij}\le 
a/\!\sqrt{N_{k-1}}$ by \eqref{A}. This proves part $ii$.
\end{proof}

Lemma \ref{lem:core} provides a lot of information about the structure of 
the matrix $X$. Near the diagonal, the matrix consists of a sequence of 
blocks of dimension $\sim N_k$ whose entries are of order $\lesssim 
b2^{-k/2}$. On the other hand, away from the diagonal, the entries of the 
matrix decay much more rapidly at a rate $\lesssim a2^{-2^{k-2}}$. We can 
therefore decompose our matrix into a block-diagonal part and a small 
remainder.
More precisely, let us partition the indices $(i,j)\in[1,n]^2$ of $X$ into 
three 
parts:
$$
	E_1 = [1,M_1]^2\cup
	\bigcup_{k\ge 1}[N_{2k},M_{2k+1}]^2,\qquad
	E_2 =
	\bigcup_{k\ge 1}[N_{2k-1},M_{2k}]^2
	\backslash E_1,
$$
and
$$
	E_3 = [1,n]^2\backslash (E_1\cup E_2),
$$
where we defined $M_k := N_{k}+N_{k}N_{k-1}$.
This partition is illustrated in Figure \ref{fig:opdecomp}.
\begin{figure}[t]
\begin{center}
\vskip.2cm
\begin{tikzpicture}
\draw (0,0) rectangle (5,-5);

\draw (1,-1) rectangle (2.3,-2.3);
\draw (3,-3) rectangle (4.3,-4.3);

\draw[fill=white] (0,0) rectangle (1.3,-1.3);
\draw[fill=white] (2,-2) rectangle (3.3,-3.3);
\draw[fill=white] (4,-4) rectangle (5,-5);

\draw (.65,-.65) node {$E_1$};
\draw (2.65,-2.65) node {$E_1$};
\draw (4.5,-4.5) node {$E_1$};

\draw (1.65,-1.65) node {$E_2$};
\draw (3.65,-3.65) node {$E_2$};

\draw (1.25,-3.75) node {$E_3$};
\draw (3.75,-1.25) node {$E_3$};

\draw (-.05,0) to (-.1,0) node[left] {\scriptsize $1$};
\draw (-.05,-1) to (-.1,-1) node[left] {\scriptsize $N_1$};
\draw (-.05,-1.3) to (-.1,-1.3) node[left] {\scriptsize $M_1$};
\draw (-.05,-2) to (-.1,-2) node[left] {\scriptsize $N_2$};
\draw (-.05,-2.3) to (-.1,-2.3) node[left] {\scriptsize $M_2$};
\draw (-.05,-3) to (-.1,-3) node[left] {\scriptsize $N_3$};
\draw (-.05,-3.3) to (-.1,-3.3) node[left] {\scriptsize $M_3$};
\draw (-.1,-4) node[left] {\scriptsize $\vdots\,$};
\draw (-.05,-5) to (-.1,-5) node[left] {\scriptsize $n$};

\draw (0,.05) to (0,.1) node[above] {\scriptsize $1$};
\draw (1,.05) to (1,.1) node[above] {\scriptsize $N_1\,$};
\draw (1.3,.05) to (1.3,.1) node[above] {\scriptsize $\,M_1$};
\draw (2,.05) to (2,.1) node[above] {\scriptsize $N_2\,$};
\draw (2.3,.05) to (2.3,.1) node[above] {\scriptsize $\,M_2$};
\draw (3,.05) to (3,.1) node[above] {\scriptsize $N_3\,$};
\draw (3.3,.05) to (3.3,.1) node[above] {\scriptsize $\,M_3$};
\draw (4.15,.1) node[above] {\scriptsize $\cdots$};
\draw (5,.05) to (5,.1) node[above] {\scriptsize $n$};
\end{tikzpicture}
\vskip.2cm
\end{center}
\caption{\small Illustration of the matrix decomposition used in the
proof of Theorem \ref{thm:latala}. The figure is drawn for clarity on a 
$\log\log$ scale.\label{fig:opdecomp}}
\end{figure}
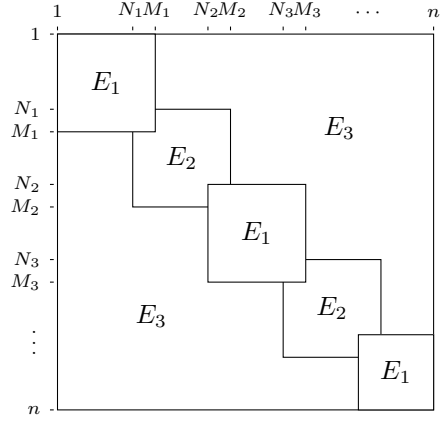

We are now ready to complete the proof of Theorem \ref{thm:latala}.

\begin{proof}[Proof of Theorem \ref{thm:latala}]
We decompose $X=U+V+W$ where 
$$
	U_{ij} :=X_{ij}1_{(i,j)\in E_1},\qquad
	V_{ij} :=X_{ij}1_{(i,j)\in E_2},\qquad
	W_{ij} :=X_{ij}1_{(i,j)\in E_3}.
$$
It suffices to bound the norms of each of these matrices.

\vskip\abovedisplayskip\noindent
\textbf{The norms of $U$ and $V$.}
The crucial point is that $U$ and $V$ are block-diagonal matrices, so that 
their norm is the maximum of the norms of the blocks. Let us denote by 
$U_1,U_2,\ldots$ the (disjoint) diagonal blocks of $U$. Then $U_1$ is a 
square matrix of dimension $M_1$ and $U_{k+1}$ is a square matrix of 
dimension $M_{2k+1}-N_{2k}+1\le N_{2k+2}$ for $k\ge 1$. Moreover, by Lemma 
\ref{lem:core}, the maximal coefficient $b_{ij}$ in $U_k$ is 
$\lesssim b2^{-k}$ for $k>1$, while we can trivially bound
$b_{ij}\lesssim b$ for coefficients in $U_1$. Thus Theorem 
\ref{thm:bvh} and the Gaussian Poincar\'e inequality \cite[Theorem 
3.20]{BLM13} yield
$$
	\mathbf{E}\|U_k\| \lesssim
	a + b2^{-k}\sqrt{\log N_{2k+2}}
	\lesssim a+b,\qquad
	\mathrm{Var}(\|U_k\|)^{1/2}\lesssim b2^{-k}.
$$
We can therefore estimate
\begin{align*}
	\mathbf{E}\|U\| =
	\mathbf{E}\max_{k}\|U_k\| 
	&\le
	\max_k\mathbf{E}\|U_k\| +
	\mathbf{E}\max_k|\,\|U_k\|-\mathbf{E}\|U_k\|\,|
\\
	&\le
	\max_k\mathbf{E}\|U_k\| +
	\sum_k
	\mathbf{E}|\,\|U_k\|-\mathbf{E}\|U_k\|\,|
\\
	&\le
	\max_k\mathbf{E}\|U_k\| +
	\sum_k \mathrm{Var}(\|U_k\|)^{1/2} \lesssim a+b.
\end{align*}
The same bound on $\mathbf{E}\|V\|$ follows from the identical argument.

\vskip\abovedisplayskip\noindent
\textbf{The norm of $W$.} We now have to show that the remaining part of 
the matrix is small. To this end, let us verify that the following holds:
$$
	\max_j b_{ij}1_{(i,j)\in E_3} \lesssim a2^{-2^{k-2}}
	\quad\mbox{whenever }M_k\le i<M_{k+1}
$$
holds for all $k\ge 1$.
Indeed, it is readily read off from the definition of $E_3$ than when
$M_k\le i<N_{k+1}$ and $(i,j)\in E_3$, then either $j\le N_k$ or
$j\ge M_{k+1}$. In either case, Lemma \ref{lem:core} shows that
$b_{ij}\lesssim a2^{-2^{k-2}}$. On the other hand, if
$N_{k+1}\le i<M_{k+1}$, then either $j\le N_k$ or
$j\ge M_{k+2}$. Again, Lemma \ref{lem:core} shows that
$b_{ij}\lesssim a2^{-2^{k-2}}$. Combining these cases yields the claim.
But we can now estimate by Theorem \ref{thm:vh}
\begin{align*}
	\mathbf{E}\|W\| 
	&\lesssim a +
	\max_{k\ge 1}
	\max_{M_k\le i<M_{k+1}} 
	\log i \,\max_j b_{ij}1_{(i,j)\in E_3}  \\
	&\lesssim a + 
	\max_{k\ge 1}
	a2^{-2^{k-2}}\log M_{k+1} \lesssim a,
\end{align*}
where we have trivially bounded $\max_{i<M_1}\log i\,
\max_j b_{ij}1_{(i,j)\in E_3}\lesssim a$.
Combining the estimates for the norms of $U,V,W$ completes the proof.
\end{proof}

\subsection{Proof of Theorem \ref{thm:schatten}}

Now that we have proved Theorem \ref{thm:schatten} in the case $p=\infty$, 
it remains to extend the conclusion to arbitrary $2\le p\le\infty$. 
Perhaps somewhat surprisingly, this does not require significant 
additional work. The reason is already visible in the formula of Lemma 
\ref{lem:gausslp}: the behavior of the expected $\ell_p$ norm 
$\mathbf{E}\|G\|_{\ell_p}$ of a Gaussian vector $G$ interpolates between 
the sharp bound on the uniform norm $\mathbf{E}\|G\|_{\ell_\infty}$ and 
the sharp bound for the moments $(\mathbf{E}\|G\|_{\ell_p}^p)^{1/p}$. We 
will show that an analogous situation occurs for Schatten norms: we will 
bound $\mathbf{E}\|X\|_{S_p}$ by decomposing the random matrix $X$ into 
two parts, one of which is controlled by the sharp bound on the 
operator norm obtained in the previous section, and the other is 
controlled by the sharp moment bound provided by Corollary 
\ref{cor:moment}.

\begin{proof}[Proof of Theorem \ref{thm:schatten}]
The lower bound on $\mathbf{E}\|X\|_{S_p}$ follows trivially from Lemma 
\ref{lem:schlp}, so it remains to prove the upper bound. We may assume 
without loss of generality that the rows and columns of $X$ have been 
permuted so that $\max_jb_{ij}$ is nonincreasing, that is, 
$b_{ij}=b_{ij}^*$. We begin by decomposing the matrix as
$$
	X = Y+Z,\qquad\quad
	Z_{ij}:=X_{ij}\mathbf{1}_{\min(i,j)\ge e^p}.
$$
This decomposition is illustrated in Figure \ref{fig:spdecomp}.
\begin{figure}[t]
\begin{center}
\vskip.2cm
\begin{tikzpicture}
\draw (0,0) rectangle (5,-5);

\draw (1,-1) rectangle (5,-5);

\draw (.55,-.55) node {$Y$};
\draw (3,-3) node {$Z$};

\draw (-.05,0) to (-.1,0) node[left] {\scriptsize $1$};
\draw (-.05,-1) to (-.1,-1) node[left] {\scriptsize $e^p$};
\draw (-.05,-5) to (-.1,-5) node[left] {\scriptsize $n$};

\draw (0,.05) to (0,.1) node[above] {\scriptsize $1$};
\draw (1,.05) to (1,.1) node[above] {\scriptsize $e^p\,$};
\draw (5,.05) to (5,.1) node[above] {\scriptsize $n$};
\end{tikzpicture}
\vskip.2cm
\end{center}
\caption{\small Matrix decomposition used in the
proof of Theorem \ref{thm:schatten}.\label{fig:spdecomp}}
\end{figure}
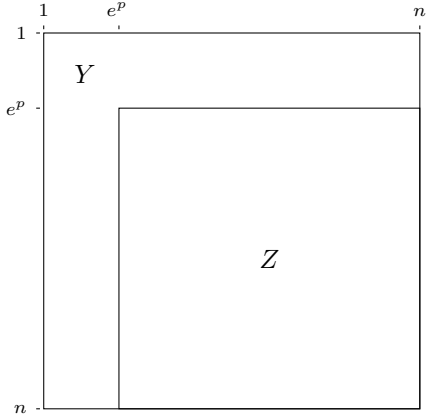

Consider first the matrix $Y$. Observe that $\mathrm{rank}(Y)\le 2e^p$,
so we have $\|Y\|_{S_p}\le 2^{1/p}e\|Y\|_{S_\infty}$. We can therefore 
apply Theorem \ref{thm:latala} and Corollary \ref{cor:expl} to estimate
$$
	\mathbf{E}\|Y\|_{S_p} \lesssim 
	\mathbf{E}\Bigg[
	\max_i\sqrt{\sum_j Y_{ij}^2}
	\Bigg] \le
	\mathbf{E}\Bigg[
	\Bigg(\sum_i\Bigg(\sum_j X_{ij}^2\Bigg)^{p/2}\Bigg)^{1/p}
	\Bigg].
$$
We now consider the matrix $Z$. Corollary \ref{cor:moment} and
Remark \ref{rem:strange} yield
$$
	\mathbf{E}\|Z\|_{S_p} \lesssim
	\Bigg(\sum_i\Bigg(\sum_j b_{ij}^2\Bigg)^{p/2}\Bigg)^{1/p}
	+ \sqrt{p}\,\Bigg(\sum_{i\ge e^p} \max_j b_{ij}^{p}\Bigg)^{1/p}.
$$
Therefore, by Corollary \ref{cor:expl}, we obtain once more
$$
	\mathbf{E}\|Z\|_{S_p} \lesssim
	\mathbf{E}\Bigg[
        \Bigg(\sum_i\Bigg(\sum_j X_{ij}^2\Bigg)^{p/2}\Bigg)^{1/p}
        \Bigg].   
$$
It remains to apply the triangle inequality
$\mathbf{E}\|X\|_{S_p}\le\mathbf{E}\|Y\|_{S_p}+\mathbf{E}\|Z\|_{S_p}$.
\end{proof}

\subsection{Proof of Corollaries \ref{cor:inf} and \ref{cor:schattentail}}

The proof of Corollary \ref{cor:inf} follows easily from
Theorem \ref{thm:schatten} and standard arguments.

\begin{proof}[Proof of Corollary \ref{cor:inf}]
It is an elementary fact \cite[section II.26]{AG93} that an infinite 
matrix $X$ defines a (necessarily unique) bounded operator on 
$\ell_2(\mathbb{N})$  
if and only if $\sup_n\|X_{[n]}\|_{S_\infty}<\infty$, where we defined the 
finite submatrices $X_{[n]}:=(X_{ij})_{i,j\le n}$.

Now note that by Theorem \ref{thm:schatten}, Remark \ref{rem:nonzeromean}, 
and the monotone convergence theorem, we have
$\mathbf{E}[\sup_n\|X_{[n]}\|_{S_\infty}]<\infty$ if and only if
$$
	\max_i\sum_j b_{ij}^2 <\infty,\qquad
	\max_{ij}b_{ij}^*\sqrt{\log i}<\infty,\qquad
	\|(a_{ij})\|_{S_\infty}<\infty.
$$
Therefore, if the latter conditions hold $X$ defines a bounded operator 
a.s. Conversely, if any of these conditions fails, then 
$\mathbf{E}[\sup_n\|X_{[n]}\|_{S_\infty}]=\infty$. By a classical zero-one 
law for Gaussian measures \cite{LS70}, it follows that in fact 
$\sup_n\|X_{[n]}\|_{S_\infty}=\infty$ a.s. Thus in this case $X$ is 
unbounded as an operator on $\ell_2(\mathbb{N})$ a.s.
\end{proof}

We now turn to the proof of Corollary \ref{cor:schattentail}.

\begin{proof}[Proof of Corollary \ref{cor:schattentail}]
We begin by proving the tail bounds. The lower bound is trivial by
Lemma \ref{lem:schlp}, so it suffices to consider the upper bound.

In the following we denote by $m$ the median of $\|X\|_{S_p}$. By the 
Gaussian isoperimetric theorem \cite[Theorem 10.17]{BLM13}, we can 
estimate for all $t\ge 0$
$$
	\mathbf{P}[\|X\|_{S_p}-m\ge t] \le
	\mathbf{P}\bigg[\sqrt{2}\max_{ij}b_{ij}\,g\ge t\bigg],
$$
where $g\sim N(0,1)$. Moreover, as the median of a nonnegative random 
variable is bounded by twice its mean (this is a simple consequence of 
Markov's inequality), we have $m\le 2\mathbf{E}\|X\|_{S_p}\le 
K\,\mathbf{E}\|X\|_{\ell_p(\ell_2)}$ for a universal constant $K$ by 
Theorem \ref{thm:schatten}.

We consider separately two cases. Suppose first that $t\ge 2m$. Then
$$
	\mathbf{P}[\|X\|_{S_p} \ge t] \le
	\mathbf{P}[\|X\|_{S_p} - m \ge t/2] \le
	\mathbf{P}\bigg[2\sqrt{2} \max_{ij}b_{ij}g \ge t\bigg] \le
	\mathbf{P}[2\sqrt{2}\|X\|_{\ell_p(\ell_2)}\ge t],
$$
where we used that $\|X\|_{\ell_p(\ell_2)}\ge X_{ij}\sim b_{ij}g$ for any 
$i,j$.

Now consider the case $t\le 2m$. Using $m\le 
K\,\mathbf{E}\|X\|_{\ell_p(\ell_2)}$, we can estimate
$$
	\mathbf{P}[\|X\|_{\ell_p(\ell_2)} \ge t/4K] \ge
	\mathbf{P}[\|X\|_{\ell_p(\ell_2)} \ge 
	\tfrac{1}{2}\mathbf{E}\|X\|_{\ell_p(\ell_2)}] \ge
	\frac{1}{4}\frac{(\mathbf{E}\|X\|_{\ell_p(\ell_2)})^2}{
	\mathbf{E}\|X\|_{\ell_p(\ell_2)}^2}
$$
by the Paley-Zygmund inequality. On the other hand, we can estimate
$$
	\mathrm{Var}(\|X\|_{\ell_p(\ell_2)})\le \max_{ij}b_{ij}^2 
	= \tfrac{\pi}{2}\max_{ij}(\mathbf{E}|X_{ij}|)^2
	\le
	\tfrac{\pi}{2}(\mathbf{E}\|X\|_{\ell_p(\ell_2)})^2
$$
by the Gaussian Poincar\'e inequality \cite[Theorem 3.20]{BLM13}.
We have therefore shown that 
$$
	\mathbf{P}[\|X\|_{\ell_p(\ell_2)} \ge t/2K] \ge
	\tfrac{1}{2\pi+4} \ge
	\tfrac{1}{2\pi+4}\mathbf{P}[\|X\|_{S_p}\ge t]
$$
for $t\le 2m$. Combining the above bounds yields the desired tail bound.

It remains to deduce the resulting bounds 
for convex functions. The lower bound is again trivial. To prove the upper 
bound, note first that
$$
	\mathbf{E}\Phi(\|X\|_{S_p}) =
	\Phi(0) + 
	\int_0^\infty \Phi'(t)\,
	\mathbf{P}[\|X\|_{S_p}\ge t]\,dt
$$
for any increasing function $\Phi$. We conclude immediately that
$$
	\mathbf{E}\Phi(\|X\|_{S_p}) \le
	C\,\mathbf{E}\Phi(C\|X\|_{\ell_p(\ell_2)})
$$
for every nonnegative increasing function $\Phi$. As compared to the bound 
stated in Corollary \ref{cor:schattentail}, we have not assumed $\Phi$ is 
convex, but we have an additional constant $C$ in front of the expectation 
on the right-hand side. To eliminate it, let $\Psi$ be an increasing 
convex function, and choose $\Phi(t) = \Psi(Ct)-\Psi(0)$. By convexity, we 
have $\Phi(t)\ge C\,\Phi(t/C)$ for all $t\ge 0$. Substituting into the 
above bound yields
$$
	C\, \mathbf{E}\Phi(\|X\|_{S_p}/C)
	\le
	\mathbf{E}\Phi(\|X\|_{S_p})
	\le
	C\, \mathbf{E}\Phi(C\|X\|_{\ell_p(\ell_2)}),
$$
which yields upon rearranging
$$
	\mathbf{E}\Psi(\|X\|_{S_p}) \le
	\mathbf{E}\Psi(C^2\|X\|_{\ell_p(\ell_2)}).
$$
Thus the claimed bound follows
for a sufficiently large universal constant $C$.
\end{proof}

\section{Extensions and complements}
\label{sec:ext}

\subsection{Non-symmetric matrices}
\label{sec:nonsym}

For simplicity, we have restricted our attention so far to symmetric matrices 
$X$ with independent entries above the diagonal. However, the analogous 
results for non-symmetric matrices can be readily deduced in complete 
generality, modulo universal constants. Let us state, for example, the 
following analogue of Corollary \ref{cor:schattentail} for non-symmetric 
matrices.

\begin{cor}
\label{cor:schattentailasym}
Let $X$ be an $n\times m$ matrix with
$X_{ij}=b_{ij}g_{ij}$, where $b_{ij}\ge 0$ and $g_{ij}$ are i.i.d.\
standard Gaussian variables for $i\le n$, $j\le m$. 
Then
$$
	\mathbf{P}[\|X\|_{\ell_p(\ell_2)}+
	\|X^*\|_{\ell_p(\ell_2)}\ge 2t] \le
	\mathbf{P}[\|X\|_{S_p}\ge t] \le
	C\,\mathbf{P}[\|X\|_{\ell_p(\ell_2)}
	+\|X^*\|_{\ell_p(\ell_2)}\ge t/C]
$$
for all $t\ge 0$ and $2\le p\le\infty$, where $C$ is a universal constant.
\end{cor}

\begin{proof}
We use a standard device that already appeared in the proof of Lemma 
\ref{lem:complext}. Let $\tilde X$ be the $(n+m)\times(n+m)$ symmetric 
matrix defined by
$$
	\tilde X =
	\begin{pmatrix}
	0 & X \\
	X^* & 0
	\end{pmatrix},
$$
and note that
$$
	\|\tilde X\|_{S_{p}}^{p} =
	\|\tilde X^2\|_{S_{p/2}}^{p/2} = 2\|X\|_{S_{p}}^{p},\qquad
	\|\tilde X\|_{\ell_p(\ell_2)}^p =
	\|X\|_{\ell_p(\ell_2)}^p +
	\|X^*\|_{\ell_p(\ell_2)}^p.	
$$
The conclusion now follows readily by applying Corollary 
\ref{cor:schattentail} to $\tilde X$.
\end{proof}

The only drawback of the approach of Corollary \ref{cor:schattentailasym} 
is that it may not capture sharp constants. As most of the bounds in this 
paper are sharp only up to universal constants to begin with, this entails 
no further loss in our main results. The exception to this statement, 
however, is the moment bound of Theorem \ref{thm:moment} which turns out 
to be highly accurate: it has the optimal constant in its leading term, a 
fact that can be of significant importance in certain applications (an 
example of such an application will be given in section \ref{sec:bdd} 
below). If we wish to obtain an analogously sharp result in the 
non-symmetric case, the proof must be modified to account for the 
non-symmetric structure. This results in the following bound.

\begin{thm}
\label{thm:momentrect}
Let $X$ be an $n\times m$ matrix with 
$X_{ij}=b_{ij}g_{ij}$, where $b_{ij}\ge 0$ and $g_{ij}$ are i.i.d.\ 
standard Gaussian variables for $i\le n$, $j\le m$. Then for $p\in\mathbb{N}$
$$
	(\mathbf{E}\|X\|_{S_{2p}}^{2p})^{\frac{1}{2p}} \le
	\Bigg(\sum_i\Bigg(\sum_j b_{ij}^2\Bigg)^p\Bigg)^{\frac{1}{2p}} + 
	\Bigg(\sum_j\Bigg(\sum_i b_{ij}^2\Bigg)^p\Bigg)^{\frac{1}{2p}}
	+ 4\sqrt{p}\,\Bigg(\sum_{i,j} b_{ij}^{2p}\Bigg)^{\frac{1}{2p}}.
$$
\end{thm}

The adaptations of the proof of Theorem \ref{thm:moment} needed to obtain 
this bound are analogous to the ones developed in \cite[section 
3.1]{BvH16} in the simpler setting considered there. The present case 
requires essentially no new ideas, but a complete rewriting of the 
combinatorial arguments of section \ref{sec:moment} in the non-symmetric 
case would be very tedious. Instead, we will briefly sketch the necessary 
modifications in the rest of this section, leaving a detailed 
verification of the proofs to the interested reader.

\begin{proof}[Sketch of proof of Theorem \ref{thm:momentrect}]
Define
$$
	\sigma_{p,1} :=
	\Bigg(\sum_i\Bigg(\sum_j 
	b_{ij}^2\Bigg)^p\Bigg)^{1/2p},\qquad
	\sigma_{p,2} :=
	\Bigg(\sum_j\Bigg(\sum_i
	b_{ij}^2\Bigg)^p\Bigg)^{1/2p}.
$$
Throughout the proof, we may assume without loss of generality that 
$\sum_{i,j}b_{ij}^{2p}=1$ (by scaling) and that 
$\sigma_{p,1}\ge\sigma_{p,2}$ (if not, replace $X$ by $X^*$).
We begin by writing
\begin{align*}
	\mathbf{E}\|X\|_{S_{2p}}^{2p} &=
	\mathbf{E}[\mathrm{Tr}[(XX^*)^p]]\\& =
	\sum_{\mathbf{u}\in[n]^p}\sum_{\mathbf{v}\in[m]^p}
	\mathbf{E}[X_{u_1v_1}X_{u_2v_1}X_{u_2v_2}X_{u_3v_2}\cdots
	X_{u_pv_p}X_{u_1v_p}].
\end{align*}
We view $u_1\to v_1\to u_2\to v_2\to \cdots\to u_p\to v_p\to u_1$ as a  
cycle of length $2p$ in the complete undirected bipartite graph with $n$ 
left vertices and $m$ right vertices; we will refer to such a bipartite 
cycle for simplicity as a \emph{bicycle}. By symmetry of the Gaussian 
distribution, it is clear that each distinct edge $\{u_k,v_k\}$ or 
$\{u_{k+1},v_k\}$ that appears in the bicycle must be traversed an even 
number of times for a term in the sum to be nonzero; we will call bicycles 
with this property \emph{even}.

The \emph{shape} $\mathbf{s}(\mathbf{u},\mathbf{v})$ of a bicycle 
$(\mathbf{u},\mathbf{v})\in[n]^p\times[m]^p$ is defined by relabeling the 
left and right vertices in order of appearance. Let $\mathcal{S}_{2p}^{\rm 
bi}$ be the set of shapes of even bicycles of length $2p$, denote by
$n_i(\mathbf{s})$ the number of distinct edges that are traversed exactly 
$i$ times by $\mathbf{s}$, and denote by $m_1(\mathbf{s})$ and 
$m_2(\mathbf{s})$ the number of distinct right and left 
vertices, respectively, that are visited by $\mathbf{s}$. Then we can 
write
$$
        \mathbf{E}\|X\|_{S_{2p}}^{2p} =
	\sum_{\mathbf{s}\in\mathcal{S}_{2p}^{\rm bi}} 
	\prod_{i\ge 1}\mathbf{E}[g^{2i}]^{n_{2i}(\mathbf{s})}
	\sum_{(\mathbf{u},\mathbf{v}):
	\mathbf{s}(\mathbf{u},\mathbf{v})=\mathbf{s}}
	b_{u_1v_1}b_{u_2v_1}\cdots b_{u_pv_p}b_{u_1v_p},
$$
where $g\sim N(0,1)$. The following is a non-symmetric analogue of
Proposition \ref{prop:holder}.

\begin{prop}
\label{prop:holderrect}
Suppose $\sum_{i,j}b_{ij}^{2p}=1$ and $\sigma_{p,1}\ge\sigma_{p,2}$. Then 
for any $\mathbf{s}\in\mathcal{S}_{2p}^{\rm bi}$
$$
	\sum_{(\mathbf{u},\mathbf{v}):
	\mathbf{s}(\mathbf{u},\mathbf{v})=\mathbf{s}}
	b_{u_1v_1}b_{u_2v_1}\cdots b_{u_pv_p}b_{u_1v_p}
	\le \sigma_{p,1}^{2m_1(\mathbf{s})}
	\sigma_{p,2}^{2(m_2(\mathbf{s})-1)}.
$$
\end{prop}

With this result in hand, the rest of the proof of Theorem 
\ref{thm:momentrect} is almost identical to that of Theorem 
\ref{thm:moment}. Indeed, repeating the same steps, we
can show that $\mathbf{E}[\mathrm{Tr}[(XX^*)^p]]\le
\mathbf{E}[\|Y\|^{2p}]$, where $Y$ is the $r\times r'$ matrix
whose entries are i.i.d.\ standard Gaussian with 
$r=\lceil\sigma_{p,2}^2+p/2\rceil$
and $r'=\lceil\sigma_{p,1}^2+p/2\rceil$. The conclusion of
Theorem \ref{thm:momentrect} now follows from a classical estimate on the 
norm of Gaussian matrices
$$
	\mathbf{E}[\|Y\|^{2p}]^{1/2p} \le \sqrt{r}+\sqrt{r'} + 2\sqrt{p}.
$$
The omitted steps are spelled out in more detail in
the proof of \cite[Theorem 3.1]{BvH16}.
\end{proof}

What remains is to prove Proposition \ref{prop:holderrect}, which proceeds 
mostly along the lines of Proposition \ref{prop:holder} but with 
additional bookkeeping issues. We will again merely sketch the necessary 
modifications to the proof of Proposition \ref{prop:holder}.

\begin{proof}[Sketch of proof of Proposition \ref{prop:holderrect}]
Let $\mathcal{G}_{m_1,m_2}$ be the set of undirected, connected 
bipartite graphs with $m_1$ right vertices and $m_2$ left vertices.
Exactly as in section \ref{sec:holder}, we 
associate to every shape $\mathbf{s}\in\mathcal{S}_{2p}^{\rm bi}$ with 
$m_1(\mathbf{s})=m_1$ and $m_2(\mathbf{s})=m_2$ a graph 
$G\in\mathcal{G}_{m_1,m_2}$ and weights $\mathbf{k}=(k_e)_{e\in E(G)}$, 
$k_e\ge 2$ with $|\mathbf{k}|=2p$ such that
$$
	\sum_{(\mathbf{u},\mathbf{v}):
	\mathbf{s}(\mathbf{u},\mathbf{v})=\mathbf{s}}
	b_{u_1v_1}b_{u_2v_1}\cdots b_{u_pv_p}b_{u_1v_p}
	\le W^{\mathbf{k}}(G).
$$
Here we define in the bipartite case
$$
	W^{\mathbf{k}}(G) :=
	\sum_{\mathbf{u}\in[n]^{m_2}}
	\sum_{\mathbf{v}\in[m]^{m_1}}
	\prod_{e\in E(G)}
	b_{u(e)v(e)}^{k_e},
$$
where we denote by $e=(u(e),v(e))$ the left and right vertices of 
the edge $e$. We must show that for every $G\in\mathcal{G}_{m_1,m_2}$
and $\mathbf{k}=(k_e)_{e\in E(G)}$, $k_e\ge 2$, $|\mathbf{k}|=2p$ we have
$$
	W^{\mathbf{k}}(G)
        \le \sigma_{p,1}^{2m_1}
        \sigma_{p,2}^{2(m_2-1)}
$$
under the assumptions of Proposition \ref{prop:holderrect}.

We begin by observing that it suffices to prove the claim only for
$G\in\mathcal{G}_{m_1,m_2}^{\rm tree}$, that is, when $G$ is a bipartite 
tree with $m_1$ right vertices and $m_2$ left vertices. The proof is 
identical to that of Lemma \ref{lem:tree}, so we do not comment on it 
further.

To obtain the requisite bound for trees, we now proceed exactly as in the 
proof of Lemma \ref{lem:bizarrobl} by iteratively pruning the leaves of 
the tree using H\"older's inequality. However, in the present case the 
coefficient matrix $b_{ij}$ is not symmetric, so we must keep track of 
which of its indices is being summed over in each step of the iteration. 
As the same edge may be pruned from either direction in the course of the 
induction, the final conclusion of the induction argument is a bound of 
the form
$$
	W^{\mathbf{k}}(G) \le
	\prod_{e\in E(G)}
	\Bigg[\sum_i\Bigg(\sum_j b_{ij}^{k_e}\Bigg)^{\frac{|\mathbf{k}|}{k_e}}
	\Bigg]^{\frac{1}{\alpha_e}}
	\Bigg[\sum_j\Bigg(\sum_i b_{ij}^{k_e}\Bigg)^{\frac{|\mathbf{k}|}{k_e}}
	\Bigg]^{\frac{1}{\beta_e}}.
$$
However, recall that in each step of the induction argument, the 
homogeneity in each variable $b_{ij}^{k_e}$ is preserved. Moreover, 
in the non-symmetric case, it is readily seen that the total number of 
sums over left and right vertices, respectively, is also preserved in each 
step of the induction argument (for example, if we bound
$\sum_{i,j,k}b_{ij}^2b_{kj}^2\le
[\sum_j(\sum_ib_{ij}^2)^2]^{1/2}
[\sum_j(\sum_k b_{kj}^2)^2]^{1/2}$, there are two sums over left vertices
and one sum over a right vertex on both sides of the inequality when we
count multiplicities given by the exponents.)
Thus the following must be valid for the final bound:
\begin{itemize}[label=\textbullet, leftmargin=*]
\item The right-hand side is $1$-homogeneous in each variable
$b_{ij}^{k_e}$. Therefore, we must have $1/\alpha_e+1/\beta_e = 
k_e/|\mathbf{k}|$ for every $e\in E(G)$.
\item There are $m_1$ sums over right vertices and $m_2$ sums over left 
vertices on the right-hand side (counting multiplicities given by the 
exponents). Therefore, we must have $\sum_{e\in E(G)} 
(|\mathbf{k}|/k_e\alpha_e+1/\beta_e) = m_1$ and
$\sum_{e\in E(G)} (|\mathbf{k}|/k_e\beta_e+1/\alpha_e) = m_2$.
\end{itemize}
Now apply H\"older's inequality to each term as we did at the end of the 
proof of Theorem \ref{thm:holder}. This yields, using
$\sum_{i,j}b_{ij}^{|\mathbf{k}|}=\sum_{i,j}b_{ij}^{2p}=1$,
$$
	W^{\mathbf{k}}(G) \le
	\prod_{e\in E(G)}
	\Bigg[
	\sum_i\Bigg(\sum_j b_{ij}^2\Bigg)^{\frac{|\mathbf{k}|}{2}}
	\Bigg]^{\frac{2}{k_e\alpha_e}}
	\Bigg[\sum_j\Bigg(\sum_i b_{ij}^2\Bigg)^{\frac{|\mathbf{k}|}{2}}
	\Bigg]^{\frac{2}{k_e\beta_e}}.
$$
But as $\sigma_{p,1}\ge\sigma_{p,2}$, we can estimate
$$
	W^{\mathbf{k}}(G) 
	=
	\sigma_{p,1}^{2\sum_e \frac{|\mathbf{k}|}{k_e\alpha_e}}
	\sigma_{p,2}^{2\sum_e \frac{|\mathbf{k}|}{k_e\beta_e}}
	\le
	\sigma_{p,1}^{2\sum_e(\frac{|\mathbf{k}|}{k_e\alpha_e}+\frac{1}{\beta_e})}
	\sigma_{p,2}^{2\sum_e(\frac{|\mathbf{k}|}{k_e\beta_e}-\frac{1}{\beta_e})}
	=
	\sigma_{p,1}^{2m_1}
	\sigma_{p,2}^{2(m_2-1)},
$$
which completes the proof.
\end{proof}

\subsection{Non-Gaussian matrices: heavy-tailed entries}
\label{sec:nong}

The main results of this paper show that for a centered Gaussian random 
matrix $X$, the distributions of the random variables $\|X\|_{S_p}$ and 
$\|X\|_{\ell_p(\ell_2)}$ are comparable in a very strong sense. This 
conclusion is surprisingly general---the comparison holds regardless of 
the variance pattern of the matrix entries---and it is far from clear from 
its simple statement why even Gaussianity of the entries should be needed. 
It is tempting to conjecture that this phenomenon should not depend on the 
entry distribution at all, but should follow from a general comparison 
principle for arbitrary random matrices with independent entries. However, 
this is not the case: as is shown by a simple example due to Seginer 
\cite[Theorem 3.2]{Seg00}, even the conclusion 
$\mathbf{E}\|X\|_{S_p}\lesssim \mathbf{E}\|X\|_{\ell_p(\ell_2)}$ is 
manifestly false when the entries of $X$ are uniformly bounded. From this 
perspective, it is not a surprise that the sharp result proves to be a 
Gaussian phenomenon, and that Gaussian analysis should play a key role 
throughout its proof.

Nonetheless, our results are by no means restricted to Gaussian entries: 
once the Gaussian result has been proved, we can extend its conclusion to 
a much larger class of distributions. To fix some ideas, consider for 
example the case where $X$ is a symmetric matrix with 
$X_{ij}=b_{ij}h_{ij}$, where $b_{ij}\ge 0$ and $h_{ij}$ are i.i.d.\ 
symmetric $\alpha$-stable random variables with $\alpha\in(1,2]$ (for
$i\ge j$). We claim 
that the conclusion of Corollary \ref{cor:schattentail} extends verbatim 
to this situation: that is, we have
$$
	\mathbf{P}[\|X\|_{\ell_p(\ell_2)}\ge t] \le
	\mathbf{P}[\|X\|_{S_p}\ge t] \le
	C\,\mathbf{P}[\|X\|_{\ell_p(\ell_2)}\ge t/C]
$$
for all $t\ge 0$ and $2\le p\le\infty$, where $C$ is a universal constant.
To prove this, it suffices to recall \cite[p.\ 176]{Fel71}
that any symmetric $\alpha$-stable random variable can be written as
$h_{ij}=g_{ij}z_{ij}$, where $g_{ij}$ is a standard Gaussian variable and 
$z_{ij}$ is a nonnegative random variable independent of $g_{ij}$. It 
therefore suffices to apply Corollary \ref{cor:schattentail} conditionally 
on $\{z_{ij}\}$ to prove the claim.

The above observations suggest the following general principle: while the 
conclusion of Theorem \ref{thm:schatten} (or Corollary 
\ref{cor:schattentail}) cannot be expected to hold in general for 
light-tailed entries, the result should extend rather generally to entry 
distributions whose tails are heavier than that of the Gaussian 
distribution. We have not proved a completely general formulation of this 
idea, and it is unclear what minimal regularity requirements are needed to 
make it precise. However, the following partial result applies to a broad 
class of heavy-tailed entry distributions, and serves as an illustration 
of how our results may be extended far beyond the Gaussian setting.

\begin{thm}
\label{thm:heavy}
Let $X$ be an $n\times n$ symmetric matrix with 
$X_{ij}=b_{ij}h_{ij}$, where $b_{ij}\ge 0$ and $h_{ij}$, $i\ge j$ are 
independent centered random variables that satisfy 
$$
	C_1p^{\beta} \le \mathbf{E}[|h_{ij}|^p]^{1/p} \le
	C_2p^{\beta}
	\qquad\mbox{for all }p\ge 2
$$ 
for some $\beta\ge \tfrac{1}{2}$.
Then for $2\le p\le\infty$
\begin{align*}
	\mathbf{E}\|X\|_{S_p} &\asymp
	\mathbf{E}\Bigg[
	\Bigg(\sum_i\Bigg(\sum_j X_{ij}^2\Bigg)^{p/2}\Bigg)^{1/p}
	\Bigg] 
	\\
	&\asymp
	\Bigg(\sum_i\Bigg(\sum_j b_{ij}^2\Bigg)^{p/2}\Bigg)^{1/p} +
	\max_{i\le e^p}\max_j b_{ij}^*(\log i)^{\beta} +
	p^\beta\Bigg(
	\sum_{i\ge e^p}\max_j {b_{ij}^{*}}^p
	\Bigg)^{1/p},
\end{align*}
where the constants depend on $C_1,C_2,\beta$ only.
\end{thm}

\begin{rem}
One may readily verify by inspection of the proof that the conclusions of 
Corollaries \ref{cor:schattentail} and \ref{cor:schattentailasym} 
extend analogously to the setting of Theorem \ref{thm:heavy}.
\end{rem}

The rest of this subsection is devoted to the proof of Theorem 
\ref{thm:heavy}. As a first step, we note that the moment assumption 
implies a tail bound.

\begin{lem}
\label{lem:paley}
Let $h$ be a random variable such that
$$
	C_1p^{\beta} \le \mathbf{E}[|h|^p]^{1/p} \le
	C_2p^{\beta}
	\qquad\mbox{for all }p\ge 2.
$$
Then there exist constants $c_1,c_2$ depending only on
$C_1,C_2,\beta$ such that
$$
	c_1e^{-t^{1/\beta}/c_1} \le
	\mathbf{P}[|h|\ge t] \le c_2e^{-t^{1/\beta}/c_2}
	\qquad\mbox{for all }t\ge 0.
$$
\end{lem}

\begin{proof}
To upper bound the tail, we note that Markov's inequality implies
$$
	\mathbf{P}[|h|\ge C_2ep^\beta] \le
	\mathbf{P}[|h|\ge e\|h\|_p] \le
	e^{-p}
$$
for all $p\ge 2$. Moreover, if we further bound $e^{-p}$ by $e^{2-p}$, 
then the inequality is trivially valid for all $p>0$. The upper tail now 
follows easily.

To lower bound the tail, we note that by the Paley-Zygmund inequality
$$
	\mathbf{P}[|h|\ge \tfrac{1}{2}C_1p^\beta] \ge
	\mathbf{P}[|h|\ge \tfrac{1}{2}\|h\|_p]
	\ge
	(1-2^{-p})^2 \bigg(\frac{\|h\|_p}{\|h\|_{2p}}\bigg)^{2p}
	\ge
	\tfrac{1}{2}
	e^{-cp}
$$
for $p\ge 2$, where $c=2\log(2^\beta C_2/C_1)>0$. Moreover, if we 
lower bound $e^{-cp}$ by $e^{-c(p+2)}$, the inequality is valid 
for all $p>0$.  The lower tail follows easily.
\end{proof}

As a consequence, we obtain the following comparison theorem.

\begin{lem}
\label{lem:cfconvex}
Let $h_i$ and $h_i'$, $i=1,\ldots,k$ be independent centered random 
variables that satisfy the moment assumption of Lemma \ref{lem:paley}
with constants $C_1,C_2,\beta$. Then there exists a constant $C$
depending only on $C_1,C_2,\beta$ such that
$$
	\mathbf{E}[f(h_1,\ldots,h_k)] \le
	\mathbf{E}[f(Ch_1',\ldots,Ch_k')]
$$
for every symmetric convex function $f:\mathbb{R}^k\to\mathbb{R}$.
\end{lem}

\begin{proof}
We begin by noting that as $h_i$ are centered, Jensen's inequality yields
$$
	\mathbf{E}[f(h_1,\ldots,h_k)] \le
	\mathbf{E}[f(h_1-\tilde h_1,\ldots,h_k-\tilde h_k)] \le
	\mathbf{E}[f(2h_1,\ldots,2h_k)]
$$
whenever $f$ is symmetric and convex, where $\tilde h_i$ are independent 
copies of $h_i$. Moreover, the random variables $h_i-\tilde h_i$
clearly satisfy the same moment assumptions as $h_i$ modulo
a universal constant. We can therefore assume without loss of generality
in the sequel that the random variables $h_i$ and $h_i'$ are 
symmetrically distributed.

To proceed, note that Lemma \ref{lem:paley} implies that 
$$
	\mathbf{P}[|h_{i}|\ge t] \le 
        c\,\mathbf{P}[c|h_{i}'|\ge t]
$$
for all $t\ge 0$ and $i$, where $c\ge 1$ is a constant that depends only 
on $C_1,C_2,\beta$. In particular, let $\delta_i\sim\mathrm{Bern}(1/c)$ be 
i.i.d.\ Bernoulli variables independent of $h$. Then
$$
	\mathbf{P}[\delta_i|h_{i}|\ge t] 
	\le
	\mathbf{P}[c|h_{i}'|\ge t]
$$
for all $t\ge 0$. By a standard coupling argument, we can couple
$(\delta_i,h_i)$ and $(h_i')$ on the same probability space
such that $\delta_i|h_i| \le c|h_i'|$ a.s. for every $i$
\cite[p.\ 127]{Lin02}.

Now let $\varepsilon_i$ be i.i.d.\ Rademacher variables. Then
$$
	\mathbf{E}[\,f(\varepsilon_1\delta_1|h_1|,
	\ldots,\varepsilon_k\delta_k|h_k|)\,|\delta,h,h'] \le
	\mathbf{E}[\,f(c\varepsilon_1|h_1'|,\ldots,
	c\varepsilon_k|h_k'|)\,|\delta,h,h']
$$
follows by convexity (it suffices to note that
$\alpha\mapsto 
\mathbf{E}[f(\alpha_1\varepsilon_1,\ldots,\alpha_k\varepsilon_k)]$ is 
convex, so its supremum over $\prod_i[-c|h_i'|,c|h_i'|]$ is attained at
one of the extreme points).
Therefore, as $h_i,h_i'$ are symetrically distributed and by
Jensen's inequality,
$$
	\mathbf{E}[f(h_1/c,\ldots,h_k/c)] \le
	\mathbf{E}[f(\delta_1 h_1,\ldots,\delta_k h_k)] \le
	\mathbf{E}[f(ch_1',\ldots,ch_k')]
$$
for every symmetric convex function $f$. This concludes the proof.
\end{proof}

We can now complete the proof of Theorem \ref{thm:heavy}.

\begin{proof}[Proof of Theorem \ref{thm:heavy}]
Let $g_{ij}$ and $\tilde g_{ij}$, $i\ge j$ be i.i.d.\ standard Gaussian 
variables, and define $h_{ij}' := g_{ij}|\tilde g_{ij}|^{2\beta-1}$.
Then it is readily verified that $h_{ij}'$ satisfies the same moment 
condition as $h_{ij}$ for each $i,j$. Therefore, by Lemma 
\ref{lem:cfconvex}, we have
$$
	\mathbf{E}\|X\|_{S_p}\asymp
	\mathbf{E}\|(b_{ij}h_{ij}')\|_{S_p},
	\qquad
	\mathbf{E}\|X\|_{\ell_p(\ell_2)}\asymp
	\mathbf{E}\|(b_{ij}h_{ij}')\|_{\ell_p(\ell_2)},	
$$
where the constants depend only on $C_1,C_2,\beta$. By applying Theorem 
\ref{thm:schatten} conditionally on $\{\tilde g_{ij}\}$, it now follows 
immediately that $\mathbf{E}\|X\|_{S_p}\asymp\mathbf{E}\|X\|_{\ell_p(\ell_2)}$.

It remains to obtain the explicit formula. To this end, let 
$\varepsilon_{ij}$, $i\ge j$ be i.i.d.\ Rademacher variables and let
$h_{ij}'' := \varepsilon_{ij}|g_{ij}|^{2\beta}$. Applying again Lemma
\ref{lem:cfconvex} yields
$$
	\mathbf{E}\|X\|_{\ell_p(\ell_2)} \asymp
	\mathbf{E}\|(b_{ij}h_{ij}'')\|_{\ell_p(\ell_2)} =
	\mathbf{E}\|(b_{ij}^{1/2\beta}g_{ij})\|_{\ell_{2\beta p}(\ell_{4\beta})}^{2\beta}.
$$
A routine application of Gaussian concentration
\cite[Theorem 5.8]{BLM13} shows that
$$
	\mathbf{E}\big[
	\|(b_{ij}^{1/2\beta}g_{ij})\|_{\ell_{2\beta p}(\ell_{4\beta})}^{2\beta}
	\big]^{1/2\beta}
	\asymp
	\mathbf{E}\|(b_{ij}^{1/2\beta}g_{ij})\|_{\ell_{2\beta 
	p}(\ell_{4\beta})},
$$
where the constant depends only on $\beta$.
The proof is concluded by a straightforward adaptation of the proof of
Corollary \ref{cor:expl}; we omit the details.
\end{proof}

\subsection{Non-Gaussian matrices: bounded entries}
\label{sec:bdd}

As was explained in the previous section, the two-sided bounds of Theorem 
\ref{thm:schatten} fail to extend to the situation where the entries of 
the matrix are light-tailed; in such cases new phenomena arise that are 
poorly understood, and the problem of obtaining two-sided bounds for 
matrices with bounded entries remains open (see \cite[section 4.2]{BvH16} 
for some discussion along these lines). Nonetheless, Gaussian results are 
still of considerable interest in this setting as they yield very good 
upper bounds on the matrix norms in many cases of practical interest. For 
example, the methods of the previous subsection may be easily adapted to 
show that if $X$ is a symmetric random matrix whose entries $X_{ij}$ are 
independent, centered, and $b_{ij}$-subgaussian, then its expected 
Schatten norms are dominated by those of the Gaussian matrix defined in 
Theorem \ref{thm:schatten}. Thus the explicit expression given in Theorem 
\ref{thm:schatten}, while not always sharp in the subgaussian case, always 
yields an upper bound on the quantities of interest.

Of particular interest in this context are the bounds of Theorems 
\ref{thm:moment} and \ref{thm:momentrect}, which not only yield very 
explicit upper bounds on the moments of Gaussian random matrices, but even 
provide sharp constants in the leading terms. The aim of this section is 
to show that for random matrices with uniformly bounded entries, these 
moment bounds admit an important refinement. While this requires only a 
minor modification of the proofs of Theorems \ref{thm:moment} and 
\ref{thm:momentrect}, we will spell out these results in some detail as 
they prove to be of considerable utility in many applications (for 
example, in the study of random graphs and in applied mathematics).

The main result of this section is the following slight refinement of
Theorem \ref{thm:bmain}.

\begin{thm}
\label{thm:bmoment}
Let $X$ be an $n\times n$ symmetric matrix with independent 
centered entries for $i\ge j$, and define the quantities
$$
	\sigma_p :=
	\Bigg(\sum_i\Bigg(\sum_j 
	\mathbf{E}[X_{ij}^2]\Bigg)^p\Bigg)^{1/2p},
	\qquad
	\sigma_p^* :=
	\Bigg(\sum_{i,j}
	\|X_{ij}\|_\infty^{2p}\Bigg)^{1/2p}.
$$
Then we have for every $p\in\mathbb{N}$
$$
	(\mathbf{E}\|X\|_{S_{2p}}^{2p})^{1/2p} \le
	2\sigma_p + C\sqrt{p}\,\sigma_p^*,
$$
where $C$ is a universal constant. Moreover, we have
$$
	\mathbf{P}[\|X\|_{S_{2p}}\ge 2\sigma_p+C\sqrt{p}\,\sigma_p^*+t] \le
	e^{-t^2/C\max_{ij}\|X_{ij}\|_\infty^2}
$$
for all $p\in\mathbb{N}$ and $t\ge 0$.
\end{thm}

A non-symmetric analogue will follow in precisely the same manner.

\begin{thm}
\label{thm:bmomentrect}
Let $X$ be an $n\times m$ matrix with independent 
centered entries. Define
$$
	\sigma_{p,1} :=
	\Bigg(\sum_i\Bigg(\sum_j 
	\mathbf{E}[X_{ij}^2]\Bigg)^p\Bigg)^{1/2p},
	\qquad
	\sigma_{p,2} :=
	\Bigg(\sum_j\Bigg(\sum_i 
	\mathbf{E}[X_{ij}^2]\Bigg)^p\Bigg)^{1/2p},
$$
and let $\sigma_p^*$ be defined as in Theorem \ref{thm:bmoment}.
Then we have for every $p\in\mathbb{N}$
$$
	(\mathbf{E}\|X\|_{S_{2p}}^{2p})^{1/2p} \le
	\sigma_{p,1} + \sigma_{p,2} + 
	C\sqrt{p}\,\sigma_p^*,
$$
where $C$ is a universal constant. Moreover, we have
$$
	\mathbf{P}[\|X\|_{S_{2p}}\ge 
	\sigma_{p,1} + \sigma_{p,2} +
        C\sqrt{p}\,\sigma_p^*+t]
	\le
	e^{-t^2/C\max_{ij}\|X_{ij}\|_\infty^2}
$$
for all $p\in\mathbb{N}$ and $t\ge 0$.
\end{thm}

The key point in these results is that the leading term only depends on 
the variances of the matrix entries, while the second term depends on 
their uniform bound. The variance-sensitive nature of the bound is crucial 
in many applications. For example, we briefly describe an application to 
random graphs.

\begin{example}
\label{ex:erdos}
Let $A$ be the adjacency matrix of a nonhomogeneous Erd\H{o}s-R\'enyi 
random graph on $n$ vertices, where each edge $\{i,j\}$ is included 
independently with probability $p_{ij}$. Applying Theorem 
\ref{thm:bmoment} with $p=\lceil \alpha\log n\rceil$ yields
$$
	\mathbf{E}\|A-\mathbf{E}A\| \le
	(\mathbf{E}\|A-\mathbf{E}A\|_{S_{2p}}^{2p})^{1/2p} 
	\le
	2e^{1/2\alpha}\sqrt{d}+Ce^{1/\alpha}\sqrt{\alpha\log n}
$$
for any $\alpha\ge 1$, where we defined
$d:=\max_i\sum_j p_{ij}$. It follows immediately that
$$
	\frac{\mathbf{E}\|A-\mathbf{E}A\|}{\sqrt{d}} \le
	2(1+o(1))
	\qquad\mbox{when }n\to\infty,~\log n=o(d).
$$
We therefore easily recover a recent result of \cite{BBK17} that was 
obtained there by a more complicated method. For the purpose of this 
application, it is important to note that both the constant $2$ and the 
condition $d\gtrsim\log n$ are in fact optimal, at least in the 
homogeneous case; see \cite{BBK17} and the references therein. Our general 
bounds therefore yield surprisingly accurate results in this example.
\end{example}

\begin{rem}
If the entries $X_{ij}$ are symmetrically distributed, one can deduce the 
results of Theorems \ref{thm:bmoment} and \ref{thm:bmomentrect} directly 
from Theorems \ref{thm:moment} and \ref{thm:momentrect} by a simple 
symmetrization argument, cf.\ \cite[Corollary 
3.6]{BvH16}. However, if we only assume that $X_{ij}$ are centered, the
symmetrization method loses an additional factor $\sqrt{2}$, while the 
sharp constant for non-symmetrically distributed entries was essential
in the application of Example \ref{ex:erdos}. The main 
observation of this section, which is implicitly contained but not stated 
in \cite{BvH16}, is that a minor modification of the proof of the 
moment bounds makes it possible to obtain the optimal constant even when 
the entries are only assumed to be  
centered. For the purpose of Example \ref{ex:erdos}, it would suffice 
to apply this idea in the simpler setting developed in \cite{BvH16}. 
\end{rem}

\begin{rem}
As was already used implicitly in Example \ref{ex:erdos}, Theorems
\ref{thm:bmoment} and \ref{thm:bmomentrect} provide rather practical 
bounds on the operator norm of $X$ by choosing $p\sim \log n$.
For example, applying Theorem \ref{thm:bmoment} with
$p=\lceil \alpha\log n\rceil$ and using $e^{1/\alpha}\le 1+2/\alpha$
for $\alpha\ge 1$, we can deduce as in the proof of 
\cite[Corollary 3.12]{BvH16} that
$$
	\mathbf{P}\Bigg[\|X\|\ge 2(1+\varepsilon)
	\max_i\sqrt{\sum_j \mathbf{E}[X_{ij}^2]} + t
	\Bigg]
	\le
	n e^{-\varepsilon t^2/C\max_{ij}\|X_{ij}\|_\infty^2}
$$
for every $t\ge 0$ and $0\le \varepsilon\le 1$ ($C$ is a universal 
constant).
This result improves on \cite[Corollary 3.12]{BvH16} in that it 
attains the optimal constant $2$ in the probability bound assuming only
that the entries $X_{ij}$ are centered, rather than symmetrically 
distributed. For non-symmetric matrices, we obtain analogously
that
\begin{multline*}
	\mathbf{P}\Bigg[\|X\|\ge (1+\varepsilon)
	\Bigg(
	\max_i\sqrt{\sum_j \mathbf{E}[X_{ij}^2]} +
	\max_j\sqrt{\sum_i \mathbf{E}[X_{ij}^2]}
	\Bigg)
	 + t
	\Bigg]
	\\ \le
	\max(n,m) e^{-\varepsilon t^2/C\max_{ij}\|X_{ij}\|_\infty^2}
\end{multline*}
in the setting of Theorem \ref{thm:bmomentrect}.
Such ``matrix concentration inequalities'' have found numerous 
applications in applied mathematics (see \cite{Tro15} and the references 
therein).
\end{rem}

The rest of this section is devoted to the proof of Theorem 
\ref{thm:bmoment}. We omit the proof of Theorem \ref{thm:bmomentrect}, 
which follows in an identical manner.

To explain the idea behind the proof, let us make a basic observation: 
while Gaussian analysis played a crucial role for the norm bound of 
Theorem \ref{thm:schatten}, the entry distribution was completely 
irrelevant in the moment bound of Theorem \ref{thm:moment}. Indeed, all 
the proof does is to compare the moments of the nonhomogeneous random 
matrix $X$ with the moments of another random matrix $Y$ that has i.i.d.\ 
entries. If $X$ is Gaussian, then we may choose $Y$ to be Gaussian as 
well, and we conclude by invoking standard bounds on the norm of Gaussian 
Wigner matrices. If the entries of $X$ are bounded and centered, exactly 
the same argument will apply provided we select an appropriate entry 
distribution for the matrix $Y$. Modulo this minor observation, the rest 
of the proof transfers readily to the present setting.

\begin{proof}[Proof of Theorem \ref{thm:bmoment}]
As in the proof of Theorem \ref{thm:moment}, we will assume without loss
of generality that the diagonal entries of the matrix vanish $X_{ii}=0$.

Our starting point is again the moment formula
$$
	\mathbf{E}[\mathrm{Tr}[X^{2p}]] =
	\sum_{\mathbf{u}\in[n]^{2p}}
	\mathbf{E}[X_{u_1u_2}X_{u_2u_3}\cdots X_{u_{2p}u_1}].
$$
Each distinct edge $\{u_k,u_{k_1}\}$ that appears in the cycle 
$\mathbf{u}$ must be traversed at least twice for that term 
in the sum to be nonzero, as we assumed the random variables $X_{ij}$ are 
centered. We call a cycle with this property \emph{admissible}. Note 
that unlike in the Gaussian setting of section \ref{sec:moment}, an 
admissible cycle is not necessarily even. This distinction will turn out 
to require only minimal modifications to the proof.

Let us denote by $\mathcal{\tilde S}_{2p}$ the set of shapes of admissible 
cycles of length $2p$. Then 
\begin{align*}
	\mathbf{E}[\mathrm{Tr}[X^{2p}]] &=
	\sum_{\mathbf{s}\in\mathcal{\tilde S}_{2p}}
	\sum_{\mathbf{u}:\mathbf{s}(\mathbf{u})=\mathbf{s}}
	\mathbf{E}[X_{u_1u_2}X_{u_2u_3}\cdots X_{u_{2p}u_1}]  \\
	&\le
	\sum_{\mathbf{s}\in\mathcal{\tilde S}_{2p}}
	\sum_{\mathbf{v}\in[n]^{m(\mathbf{s})}}
	\prod_{e\in E(G(\mathbf{s}))}
	\mathbf{E}\big[|X_{v(e)}|^{k_e(\mathbf{s})}\big],
\end{align*}
where $m(\mathbf{s})$ is the number of distinct vertices visited by 
$\mathbf{s}$, $G(\mathbf{s})\in\mathcal{G}_{m(\mathbf{s})}$ is the
graph whose edges are given by 
$E(G(\mathbf{s}))=\{\{s_1,s_2\},\{s_2,s_3\},\ldots,\{s_{2p},s_1\}\}$, and
$k_e(\mathbf{s})$ is the number of times edge
$e\in E(G(\mathbf{s}))$ is traversed by $\mathbf{s}$.

We can now easily adapt the proof of Theorem \ref{thm:holder} to show that 
for every $\mathbf{s}\in\mathcal{\tilde S}_{2p}$, there exist
$k_1',\ldots,k_m'\ge 2$ with $\sum_{\ell=2}^m k_\ell'=2p$ such that
$$
	\sum_{\mathbf{v}\in[n]^{m(\mathbf{s})}}
        \prod_{e\in E(G(\mathbf{s}))}
        \mathbf{E}\big[|X_{v(e)}|^{k_e(\mathbf{s})}\big]
	\le
	\prod_{\ell=2}^{m(\mathbf{s})} 
	\Bigg( \sum_i \Bigg( \sum_j
	\mathbf{E}[|X_{ij}|^{k_\ell'}]
	\Bigg)^{\frac{2p}{k_\ell'}}\Bigg)^{\frac{k_\ell'}{2p}}.
$$
Indeed, if we simply replace $b_{ij}^k$ by $\mathbf{E}[|X_{ij}|^k]$ 
throughout, Lemma \ref{lem:tree} extends directly to the 
present setting (modulo a trivial application of Jensen's inequality
$\mathbf{E}[|X_{ij}|^{k(i)}]\le \mathbf{E}[|X_{ij}|^{k_m'}]^{k(i)/k_m'}$
in the second and in the last equation display of the proof), while 
Lemma \ref{lem:bizarrobl} can be applied verbatim. We further estimate
$$
	\sum_i \Bigg( \sum_j
        \mathbf{E}[|X_{ij}|^{k_\ell'}]\Bigg)^{\frac{2p}{k_\ell'}} 
\!	\le
	\sum_i \Bigg( \sum_j
        \mathbf{E}[|X_{ij}|^2]\Bigg)^{\frac{2p}{k_\ell'}} 
\!	\max_j \|X_{ij}\|_\infty^{\frac{2p(k_\ell'-2)}{k_\ell'}}
\!	\le
	\sigma_p^{\frac{4p}{k_\ell'}}
	(\sigma_p^*)^{2p-\frac{4p}{k_\ell'}}
$$
using H\"older's inequality, where $\sigma_p,\sigma_p^*$ are as defined in
Theorem \ref{thm:bmoment}.

By rescaling the matrix $X$, we may assume without loss of generality
that $\sigma_p^*=1$. Putting together the above estimates, we have shown 
in this case that
$$
	\mathbf{E}[\mathrm{Tr}[X^{2p}]] \le
	\sum_{\mathbf{s}\in\mathcal{\tilde S}_{2p}}
	\sigma_p^{2(m(\mathbf{s})-1)}.
$$
Now let $Y$ be an $r\times r$ symmetric matrix with entries
$Y_{ij}=(\delta_{ij}-\mathbf{E}[\delta_{ij}])/\mathrm{Var}(\delta_{ij})^{1/2}$, 
where $\delta_{ij}\sim \mathrm{Bern}(\tfrac{1}{4})$ are i.i.d.\ 
Bernoulli random variables for $i\ge j$. The key point of this choice of 
distribution is that $\mathbf{E}[Y_{ij}^k]\ge 1$ for every integer $k\ge 
2$, as may be verified by a simple explicit computation. We therefore 
obtain
$$
	\mathbf{E}[\mathrm{Tr}[Y^{2p}]] =
	\sum_{\mathbf{s}\in\mathcal{\tilde S}_{2p}}
	\sum_{\mathbf{u}:\mathbf{s}(\mathbf{u})=\mathbf{s}}
	\mathbf{E}[Y_{u_1u_2}Y_{u_2u_3}\cdots Y_{u_{2p}u_1}]  
	\ge \sum_{\mathbf{s}\in\mathcal{\tilde S}_{2p}}
	\frac{r!}{(r-m(\mathbf{s}))!}.
$$
It now follows exactly as in the proof of Theorem \ref{thm:moment} that
$$
	\mathbf{E}[\mathrm{Tr}[X^{2p}]] \le
	\mathbf{E}[\|Y\|^{2p}]
        \quad\mbox{for}\quad
        r=\lfloor\sigma_p^2\rfloor+p+1.
$$
To control $\|Y\|$, it remains to invoke a norm bound for Wigner 
matrices with non-symmetrically distributed entries. Such a bound can be 
found, for example, in \cite{PS07}, where it is shown that
$\mathbf{E}\|Y\|\le 2\sqrt{r}+C'$ for a universal constant $C'$.
Applying Talagrand's concentration inequality \cite[Theorem 6.10]{BLM13}, 
we can conclude that
$$
	\mathbf{E}[\|Y\|^{2p}]^{1/2p}\le
	2\sqrt{r}+C''\sqrt{p}
$$
for a universal constant $C''$. Therefore
$$
	\mathbf{E}[\mathrm{Tr}[X^{2p}]]^{1/2p} \le
	2\sigma_p+C\sqrt{p},
$$
concluding the proof of the moment bound. The tail bound
follows immediately by another application of Talagrand's
concentration inequality.
\end{proof}

\begin{rem}
By a further adaptation of the method of proof of Theorem 
\ref{thm:bmoment}, one can prove general bounds that apply even to 
unbounded entries: for example, if $X$ is a symmetric matrix with 
independent centered entries for $i\ge j$, we have
$$
	(\mathbf{E}\|X\|_{S_{2p}}^{2p})^{1/2p} \le
	2\Bigg(
	\mathbf{E}
	\Bigg[\sum_i
	\Bigg(\sum_j 
	X_{ij}^2\Bigg)^p\Bigg]\Bigg)^{1/2p} +
	C\sqrt{p}\,
	\Bigg(\mathbf{E}\Bigg[\sum_i\max_j X_{ij}^{2p}\Bigg]
	\Bigg)^{1/2p}
$$
for all $p\in\mathbb{N}$. This could be viewed as a random matrix analogue 
of Rosenthal's inequality \cite[Theorem 15.10]{BLM13}.
Unfortunately, this inequality typically does not give the correct scaling 
in $p$ when the entries are unbounded (consider, for example, the
case when the random variables are Gaussian). Nonetheless, such 
inequalities can be useful as they apply to very general situations 
without any additional work.

To prove this inequality, one simply 
estimates using H\"older's inequality
$$
	\sum_i \Bigg( \sum_j
        \mathbf{E}[|X_{ij}|^{k_\ell'}]\Bigg)^{\frac{2p}{k_\ell'}} 
	\le
	\Bigg(\mathbf{E}\Bigg[\sum_i\Bigg(\sum_j X_{ij}^2\Bigg)^p\Bigg]
	\Bigg)^{\frac{2}{k_\ell'}}
	\Bigg(\mathbf{E}\Bigg[\sum_i\max_j X_{ij}^{2p}\Bigg]
	\Bigg)^{1-\frac{2}{k_\ell'}}
$$
instead of the corresponding estimate in the proof of Theorem 
\ref{thm:bmoment}.
\end{rem}

\subsection*{Acknowledgments}
R.L.\ was supported by the National Science Centre of Poland under grant 
number 2015/18/A/ST1/00553. R.v.H.\ was supported in part by NSF grant 
CAREER-DMS-1148711 and by the ARO through PECASE award W911NF-14-1-0094. 
P.Y.\ was supported by grant ANR-16-CE40-0024-01. P.Y.\ would like to
thank O.\ Gu\'edon for introducing him to the questions investigated
in this paper.

This work was conducted while the authors were in residence at the 
Mathematical Sciences Research Institute in Berkeley, California, 
supported by NSF grant DMS-1440140. The hospitality of MSRI and of the 
organizers of the program on Geometric Functional Analysis is gratefully 
acknowledged.


\end{document}